\documentclass[12pt, a4paper, leqno]{amsart}

\usepackage{amsmath}
\usepackage{amsfonts}
\usepackage{amssymb}
\usepackage{mathrsfs}
\usepackage{color}
\usepackage[T1]{fontenc} 
\usepackage{url}

\usepackage[unicode=true, pdfusetitle,
 bookmarks=true,bookmarksnumbered=true,bookmarksopen=true,bookmarksopenlevel=1,
 breaklinks=false,pdfborder={0 0 1},backref=false,colorlinks=false]{hyperref}
\hypersetup{pdfstartview=FitH}

\numberwithin{equation}{section}

\newtheorem{theorem}{Theorem}[section]
\newtheorem{definition}[theorem]{Definition}
\newtheorem{lemma}[theorem]{Lemma}
\newtheorem{proposition}[theorem]{Proposition}
\newtheorem{corollary}[theorem]{Corollary}

\theoremstyle{remark}
\newtheorem{remark}[theorem]{Remark}

\theoremstyle{plain}

\providecommand{\loc}{{\ensuremath{\mathrm{loc}}}}

\newcommand{\W}{\mathcal{W}^{\alpha}_{\alpha_1,\alpha_2}}
\newcommand{\w}{\textbf{\textit{w}}}

\newcommand{\px}{{p(\cdot)}}
\newcommand{\qx}{{q(\cdot)}}
\newcommand{\ux}{{u(\cdot)}}
\newcommand{\sx}{{s(\cdot)}}

\newcommand{\B}{B^{\textbf{\textit{w}}}_{\px,\qx}}

\newcommand{\BN}{\mathcal{N}^{\textbf{\textit{w}}}_{\px,\ux,\qx}}
\newcommand{\FE}{\mathcal{E}^{\textbf{\textit{w}}}_{\px,\ux,\qx}}

\newcommand{\bns}{n^{\textbf{\textit{w}}}_{\px,\ux,\qx}}
\newcommand{\fes}{e^{\textbf{\textit{w}}}_{\px,\ux,\qx}}

\newcommand{\Nz}{\ensuremath{\mathbb{N}_0}}
\newcommand{\R}{\mathbb{R}}
\newcommand{\N}{\mathbb{N}}
\newcommand{\Z}{\mathbb{Z}}
\newcommand{\cS}{\mathcal{S}}

\newcommand{\cN}{\mathcal{N}}
\newcommand{\Rn}{{\mathbb{R}^n}}
\newcommand{\Zn}{{\mathbb{Z}^n}}

\DeclareMathOperator{\supp}{supp}

\def\esssup{\operatornamewithlimits{ess\,sup}}
\def\essinf{\operatornamewithlimits{ess\,inf}}
\newcommand{\PPlog}{\mathcal{P}^{\log}(\Rn)}
\newcommand{\PP}{\mathcal{P}(\Rn)}

\begin{document}

\title[Variable Besov-Morrey spaces]{Variable exponent Besov-Morrey spaces}

\author[A. Almeida]{Alexandre Almeida$^*$}
\address{Center for R\&D in Mathematics and Applications, Department of Mathematics, University of Aveiro, 3810-193 Aveiro, Portugal}
\email{jaralmeida@ua.pt}

\author[A. Caetano]{Ant\'{o}nio Caetano}
\address{Center for R\&D in Mathematics and Applications, Department of Mathematics, University of Aveiro, 3810-193 Aveiro, Portugal}
\email{acaetano@ua.pt}

\thanks{$^*$ Corresponding author.}
\thanks{The research was partially supported by the project ``\emph{Smoothness Morrey spaces with variable exponents}'' approved under the agreement `Projektbezogener Personenaustausch mit Portugal -- A{\c{c}}{\~o}es Integradas Luso-Alem{\~a}s' / DAAD-CRUP. It was also supported by FCT through CIDMA - Center for Research and Development in Mathematics and Applications, within project UID/MAT/04106/2019.}
\thanks{\copyright 2019. Licensed under the CC BY-NC-ND 4.0 license http://creativecommons.org/licenses/by-nc-nd/4.0/}

\date{\today}



\begin{abstract}
In this paper we introduce Besov-Morrey spaces with all indices variable and study some fundamental properties.
This includes a description in terms of Peetre maximal functions and atomic and molecular decompositions.
This new scale of non-standard function spaces requires the introduction of variable exponent mixed Morrey-sequence spaces, which
in turn are defined within the framework of semimodular spaces. In particular, we obtain a convolution inequality involving special radial kernels,
which proves to be a key tool in this work.
\end{abstract}

\maketitle


\section{Introduction}\label{sec:intro}

In recent years there has been an increase of interest in studying smoothness spaces based on the idea of Morrey spaces. The so-called Besov-Morrey spaces $\mathcal{N}^s_{p,u,q}$ and the spaces $B^{s,\tau}_{p,q}$ of Besov-type are examples of that. Roughly speaking, these spaces differ from the classical Besov spaces $B^s_{p,q}$ by the introduction of a local control as on the Morrey scale $M_{p,u}$, though the way that control is implemented differs depending on the scale of spaces considered. The spaces $\mathcal{N}^s_{p,u,q}$ were introduced by Kozono and Yamazaki \cite{KY94} and developed later by Mazzucato \cite{Mazz03} in connection with the study of Navier-Stokes equations. We refer to the papers \cite{Ros13} and \cite{Saw09} and the surveys \cite{Sick12,Sick13} for various properties and historical remarks. The spaces $B^{s,\tau}_{p,q}$ seem to have been first introduced by El Baraka \cite{ElBar06,ElBar02}, but for a systematic study and further references see \cite{YZickT10} and the surveys \cite{Sick12,Sick13}. In particular, it is known that the scales $\mathcal{N}^s_{p,u,q}$ and $B^{s,\tau}_{p,q}$ are different whenever $q$ is finite. Nevertheless, both include the classical Besov spaces $B^s_{p,q}$ as a particular case.

The theory of Morrey spaces goes back to Morrey \cite{Mor38} who considered related integral inequalities in connection with regularity properties of solutions to nonlinear elliptic equations. A further development of such theory was carried out in \cite{Cam64} leading to a wider class of spaces known as Morrey-Campanato spaces. The Morrey scale $M_{p,u}$ refines the usual scale of $L_p$ spaces since $M_{p,p}=L_p$. Morrey spaces and Besov spaces have important applications in the study of heat and Navier-Stokes equations, see \cite{L-R07,L-R12,L-R16}, \cite{Tri13}. As deeply discussed in \cite{L-R16}, the later equations constitute a challenging problem and a better understanding may require new features in new function spaces.

In the last decades function spaces with variable exponents have attracted the attention of many researchers not only by theoretical reasons but also by the role played by such spaces in some applications, including the modeling of electrorheological fluids \cite{Ruz00}, image restoration \cite{CheLR06,LiLP10}, PDE and the calculus of variations involving nonstandard growth conditions \cite{AceM01,Fan07,LZhang13}. We refer to the monographs \cite{C-UF13}, \cite{DHHR11} and \cite{KMRS16b} for the main properties, historical remarks and harmonic analysis results on variable exponent Lebesgue, Sobolev and Morrey-Campanato spaces (and others). Variable exponent Morrey spaces $M_{\px,\ux}$ were introduced in \cite{AlmHS08} and independently in \cite{KokM08,KokM10} and \cite{MizShi08}. The Besov scale $B^{\sx}_{\px,\qx}$ with all the indices variable was introduced by the first named author and H\"ast\"o \cite{AlmH10} through the consideration of the mixed Lebesgue-sequence spaces $\ell_{\qx}(L_\px)$. In particular, the exponent $q$ was allowed to depend on the space variable like $p(x)$. A simpler case occurs when $q$ is a constant since in that case $\ell_{q}(L_\px)$ becomes an iterated space and hence preserves the properties of the basis space $L_\px$. Note that the consideration of all the indices variable allow us to study important properties involving the interaction among all the parameters as it occurs, for example, in trace properties.

Recently Drihem \cite{Dri15b,Dri15c} and Yang, Yuan and Zhuo \cite{YZY15a} studied variable exponent versions $B^{\sx,\tau(\cdot)}_{\px,\qx}$ and $B_{\px,\qx}^{\sx,\phi}$ of spaces of Besov-type (see also the even more recent paper \cite{WYYZ18}). These two scales cover both the variable Besov spaces $B^{\sx}_{\px,\qx}$ and the Besov-type spaces $B^{s,\tau}_{p,q}$. On the other hand, Fu and Xu \cite{FuXu11} considered spaces $\mathcal{N}^s_{\px,\ux,q}$, therefore allowing $p$ and $u$ to vary from point to point but keeping $s$ and $q$ constant. Nevertheless, up to authors' knowledge, a full variable generalization of the Besov-Morrey scale  $\mathcal{N}^s_{p,u,q}$ is still not available. Probably, the main issue in dealing with such a generalization $\mathcal{N}^\sx_{\px,\ux,\qx}$ with all indices variable has to do  with the role played by the exponent $q$ when it is non constant. Even considering the partial variable generalization  $\mathcal{N}^s_{\px,\ux,q}$ from \cite{FuXu11}, quite some results presented there are unreliable, as some of the arguments used are wrong (see the end of Section~\ref{sec:besov-morrey} below for more details).

In this paper we present an appropriate setting to the full scale $\mathcal{N}^\sx_{\px,\ux,\qx}$ of Besov-Morrey spaces with variable smoothness and integrability. Actually, based on the knowledge we have acquired in \cite{AlmC15a, AlmC16b}, the new setting even works for $2$-microlocal versions $\BN$. As discussed in the survey paper \cite{Sick12} for the constant exponents case, there is no coincidence between this scale and the scales $B^{\sx,\tau(\cdot)}_{\px,\qx}$ and $B_{\px,\qx}^{\sx,\phi}$ mentioned above.  

We introduce mixed variable Morrey-sequence spaces $\ell_{\qx}\big(M_{\px,\ux}\big)$ within the framework of semimodular spaces. For non constant $q$, these spaces have a mixed structure like $\ell_{\qx}(L_\px)$, which prevent us to use maximal function inequalities as explained in \cite[Section~4]{AlmH10}. The lack of such tools does not allow to fit our spaces in the general axiomatic approach given in \cite{HedNet07}. On the other hand, the approach  proposed in \cite{LSUYY13}, based on Peetre maximal functions, does not help us as well, since it does not cover the cases when $q$ varies from point to point. Apart from possible applications to the study of Navier-Stokes equations, the new spaces $\BN$ unify into one single scale various function spaces treated separately by different authors in the recent years.

The new Morrey-sequence spaces introduced in this paper refine the Lebesgue-sequence scale introduced in \cite{AlmH10} since $\ell_{\qx}\big(M_{\px,\px}\big)= \ell_{\qx}(L_\px)$. This implies that the new variable Besov-Morrey scale $\BN$ includes the variable Besov scale $\B$ as a particular case, namely $\mathcal{N}^{\textbf{\textit{w}}}_{\px,\px,\qx}=\B$.

In this article much of the substantial work focuses in the study of various fundamental properties of the mixed Morrey-sequence spaces themselves. This constitutes a first part of the paper and it includes the study of the semimodular and quasinormed structures (Section~\ref{sec:mixedmorrey}) and a key convolution inequality that should replace maximal inequalities in this context (Section~\ref{sec:convolution}). The study of the full variable exponent Besov-Morrey scale is developed in a second part. After introducing the spaces $\BN$ in Section~\ref{sec:besov-morrey}, we present a characterization in terms of Peetre maximal functions (in Section~\ref{sec:PeetreMaxFunc}) and establish atomic and molecular representations for these spaces (in Section~\ref{sec:atomic}). As a by-product of the discrete decompositions, we also establish the embeddings $\cS \hookrightarrow \BN \hookrightarrow \cS'$ and show the completeness of the new spaces $\BN$.



\section{Preliminaries}\label{sec:prelim}

As usual, we denote by $\mathbb{R}^{n}$ the $n$-dimensional real
Euclidean space, $ \N$ the collection of all natural numbers and
$\N_{0}= \N\cup \{0\}$. If $a$ is a positive number then $\lfloor a \rfloor$ denotes its integer part. We write $B(x,r)$ for the open ball in
$\mathbb{R}^{n}$ centered at $x\in \mathbb{R}^{n}$ with radius $r>0$.
We use $c$ as a generic positive constant, i.e.\ a constant whose
value may change with each appearance. The expression $f
\lesssim g$ means that $f\leq c\,g$ for some independent constant
$c$, and $f\approx g$ means $f \lesssim g \lesssim f$.

Throughout the paper we denote by $\mathcal{M}(\Rn)$ the family of all complex or extended real-valued measurable functions on $\Rn$, and by $\mathcal{M}_0(\Rn)$ the {family} consisting of all those functions from $\mathcal{M}(\Rn)$ which are finite a.e. (with respect to the Lebesgue measure in $\Rn$).

%

\subsection{Semimodular spaces}

We refer to the monographs \cite{DHHR11} and \cite{Mus83} for an exposition on (semi)modular spaces. For the sake of completeness we give here a brief review on this subject.

\begin{definition}\label{def:modular}
Let $X$ be a (real or complex) vector space. A functional $\varrho: X \rightarrow [0,\infty]$ is called a semimodular on $X$ if:
\begin{enumerate}
\item[(i)] $\varrho(0_X)=0$;
\item[(ii)] $\varrho(\lambda x)=0$ for all $\lambda>0$ implies $x=0_X$;
\item[(iii)] $\varrho(\lambda x)=\varrho(x)$ for all $x\in X$ and all $\lambda$ with $|\lambda|=1$;
\end{enumerate}
The function $\varrho$ is called a modular if, in addition,
\begin{enumerate}
\item[(iv)] $\varrho(x)=0$ implies $x=0_X$.
\end{enumerate}
We say that $\varrho$ is left-continuous if
\begin{enumerate}
\item[(v)] $\lim\limits _{\lambda \to 1^-}\varrho(\lambda x)=\varrho(x)$ for all $x\in X$.
\end{enumerate}
If there exists $A\ge 1$ such that
\begin{enumerate}
\item[(vi)] $\varrho(\theta x+ (1-\theta)y)\leq A \left[ \theta \varrho(x)+ (1-\theta)\varrho(y)\right]$ for all $x,y\in X$ and $0\le \theta\le 1$,
\end{enumerate}
then $\varrho$ is said to be quasiconvex (convex if one can take $A=1$).
\end{definition}

We consider also the functional $\|\cdot\|_\varrho: X \to [0,\infty]$ given by
$$
\|x\|_\varrho :=\inf\left\{ \lambda >0 : \varrho(x/\lambda) \leq 1\right\}
$$
(assuming the usual convention $\inf \varnothing = \infty$).

If the semimodular $\varrho$ is left-continuous then we have
$$\|x\|_\varrho \leq 1 \ \ \ \text{if and only if} \ \ \ \ \varrho(x)\leq 1.$$
This fact, referred in \cite{DHHR11} as the \emph{unit ball property}, is very useful from the technical point of view since it allows one to skip working with the complicated structure of the quasinorm directly, in many situations of interest.

If $\varrho$ is a (quasi)convex semimodular, then
\begin{equation}\label{def:modular-space}
X_\varrho:=\{x\in X: \varrho(\lambda x)<\infty \ \ \text{for some} \ \ \lambda>0\}
\end{equation}
is a vector subspace of $X$ and $\|\cdot\|_\varrho$ defines a (quasi)norm on it ($X_\varrho$ is called a \emph{semimodular space}).

\subsection{Variable exponents}

By $\PP$ we denote the set of all measurable functions $p:\Rn
\rightarrow (0,\infty]$ (called \textit{variable exponents}) which
are essentially bounded away from zero.  For a measurable set $E\subset \Rn$
and $p\in \PP$,  we denote $p_E^+ :=\esssup_E p(x)$ and
$p_E^-:=\essinf_E p(x)$. For simplicity we use the abbreviations $p^+:=p_\Rn^+$ and
$p^-:=p_\Rn^-$.

The \emph{variable exponent Lebesgue space} $L_\px:=L_{\px}(\Rn)$ is the
{family of (equivalence classes of) functions $f \in \mathcal{M}(\Rn)$} such that
\begin{equation}\label{def:Lpmod}
\varrho_{\px}(f/\lambda):=\int_\Rn \phi_{p(x)}\left(\frac{|f(x)|}{\lambda}\right)\, dx
\end{equation}
is finite for some $\lambda>0$, where
\begin{equation}\label{def:Lpmod-aux}
\phi_{p(x)}(t) :=
\begin{cases}
t^{p(x)} & \text{ if } p(x)\in (0,\infty), \\
0 & \text{ if } p(x)=\infty \text{ and } t\in [0,1], \\
\infty & \text{ if } p(x)=\infty \text{ and } t\in(1,\infty]. \\
\end{cases}
\end{equation}

It is {clear that $\varrho_\px$ makes sense in $\mathcal{M}(\Rn)$ and it is known that it} defines a semimodular in the vector space {(of equivalence classes of functions in)} $\mathcal{M}_0(\Rn)$ and that $L_{\px}$ becomes a quasi-Banach space with respect to the quasinorm
\begin{align*}
    \| f|L_\px\| &:= \inf \left\{ \lambda>0 : \varrho_{\px}\left(f/\lambda\right) \leq 1\right\}.
\end{align*}
This functional defines a norm when $p^-\geq 1$. Simple calculations show that
\begin{equation}\label{Lp-t-power}
\left\| |f|^t \,|\, L_{\px/t}\right\| = \big\| f\,|\,L_{\px}\big\|^t \,, \ \ \ \ t\in(0,\infty).
\end{equation}
If $p(x)\equiv p\in(0,\infty]$ is constant, then $L_{\px}=L_{p}$ is the classical Lebesgue space.


Although the connection between the semimodular and the quasinorm is not so simple as in the constant exponent case, for variable exponents $p\in\PP$ we always have
$$\varrho_\px(f) \leq 1 \ \ \ \ \ \text{if and only if} \ \ \ \ \ \| f\,|L_\px\| \leq 1$$
due to the left-continuity of the semimodular  $\varrho_\px$.

It is worth noting that $L_{\px}$ has the lattice property and that the assertions $f\in L_\px$ and $\|f\,|\, L_\px\|<\infty$ are equivalent for any {$f\in \mathcal{M}(\Rn)$}. With an absolute constant, H\"older's inequality holds in the form
\[
\int_{\Rn} |f(x)g(x)|\,dx \leq 2\,\| f\,|L_\px\| \, \| g\,|L_{p'(\cdot)}\|
\]
for $p\in\PP$ with $p^-\ge 1$, where $p'$ denotes the conjugate exponent of $p$ defined pointwisely by $\tfrac{1}{p(x)}+\tfrac{1}{p'(x)}=1, \ \ x\in\Rn$.
These and other fundamental properties of the spaces $L_\px$, at least in the case $p^-\geq 1$, can be found in \cite{KR91} and in the recent monographs \cite{C-UF13}, \cite{DHHR11}. The definition above of $L_{\px}$ using the semimodular $\varrho_\px$ is taken from \cite{DHHR11}.

\subsection{Morrey spaces with variable exponents}

For $p,u\in \PP$ with $0<p^-\leq p(x)\leq u(x) \leq \infty$, the variable exponent Morrey space $M_{\px,\ux}:=M_{\px,\ux}(\Rn)$ consists of all functions {$f\in\mathcal{M}(\Rn)$} with finite quasinorm
\begin{equation}\label{def:morrey-norm}
\big\| f \,| M_{\px,\ux}\big\|:= \sup_{x\in\Rn,r>0} r^{\frac{n}{u(x)}-\frac{n}{p(x)}} \|f\,\chi_{B(x,r)} \,| L_\px\|.
\end{equation}
By the definition of the $L_\px$ quasinorm, we see that \eqref{def:morrey-norm} can also be written as
$$ \big\| f \,| M_{\px,\ux}\big\| = \sup_{x\in\Rn,r>0} \inf\left\{\lambda>0: \varrho_\px\left( r^{\frac{n}{u(x)}-\frac{n}{p(x)}} \tfrac{f}{\lambda}\,\chi_{B(x,r)} \right) \leq 1\right\}. $$

Note that variable exponent Morrey spaces were introduced in \cite{AlmHS08} in the Euclidean case and independently in \cite{KokM08,KokM10} in the more general setting of quasimetric spaces. Embedding results and equivalent norms were given in \cite{AlmHS08} in the case of bounded domains and $p(x)\geq 1$, under the $\log$-H\"older continuity of the exponents. Morrey spaces with variable exponents also appeared in \cite{Ohno08} and \cite{GulHS10}, the latter considering even generalized versions. We also mention \cite{Has09} where a partition norm was introduced in variable exponent Morrey spaces, which allows to pass from local results to corresponding global ones on $\Rn$. Not all definitions found in the literature are equivalent, though. The definition above with the quasinorm \eqref{def:morrey-norm} follows the approach from \cite{GulS13}, where the boundedness of various classical operators in such spaces was studied on unbounded domains.

Like in the $L_\px$ case, simple calculations show that
\[
\left\| |f|^t \,|\, M_{\px/t,\ux/t}\right\| = \big\| f\,|\,M_{\px,\ux}\big\|^t \,, \ \ \ \ t\in(0,\infty).
\]

As in the constant exponent setting, the variable exponent Morrey scale includes the variable Lebesgue spaces as a particular case. This property is formulated in Lemma~\ref{lem:coincides} below.

First we observe that, in a certain sense, the supremum and the infimum involved in the definition of the Morrey quasinorm may interchange with each other.

\begin{lemma}\label{lem:inf-sup}
Let $p\in\PP$, $v:\Rn\to [0,\infty)$ be measurable and $g$ be a complex or extended real-valued function on $\Rn\times\R^+\times\Rn$ such that, for any $(x,r)\in \Rn\times\R^+$, $g(x,r,\cdot)$ is measurable on $\Rn$. Then
\begin{equation}\label{eq:inf-sup}
\sup_{x\in\Rn,r>0} \inf\left\{\lambda>0: \varrho_\px\left( \frac{g(x,r,\cdot)}{\lambda^{v(\cdot)}} \right) \leq 1\right\} = \inf\left\{\lambda>0: \sup_{x\in\Rn,r>0} \varrho_\px\left( \frac{g(x,r,\cdot)}{\lambda^{v(\cdot)}} \right) \leq 1\right\}.
\end{equation}
\end{lemma}

\begin{proof}
For simplicity, let $a$ and $b$ denote, respectively, the right-hand side and the left-hand side of \eqref{eq:inf-sup}. We show first that $a\leq b$ (assuming the latter finite, otherwise there is nothing to prove). Given any $\varepsilon>0$, for all $x\in\Rn$ and $r>0$ we have
$$
b + \varepsilon > \inf\left\{\lambda>0: \varrho_\px\left( \frac{g(x,r,\cdot)}{\lambda^{v(\cdot)}} \right) \leq 1\right\}.
$$
Thus
$$
\varrho_\px\left( \frac{g(x,r,\cdot)}{(b + \varepsilon)^{v(\cdot)}} \right) \leq 1
$$
by the monotonicity of $\varrho_\px$. Passing to the supremum on $x$ and $r$, we get $b+\varepsilon \geq a$. Hence the claim follows by letting $\varepsilon \to 0$.
We show now that $a\geq b$ (for $a<\infty$). Given any $\varepsilon>0$, we have
$$
\sup_{x\in\Rn,r>0} \varrho_\px\left( \frac{g(x,r,\cdot)}{(a+\varepsilon)^{v(\cdot)}} \right) \leq 1
$$
since $\displaystyle \sup_{x\in\Rn,r>0} \varrho_\px$ is order preserving. Consequently, for all $x\in\Rn$ and $r>0$,
$$
\varrho_\px\left( \frac{g(x,r,\cdot)}{(a+\varepsilon)^{v(\cdot)}}  \right) \leq 1\,,
$$
which implies
$$
\inf\left\{\lambda>0: \varrho_\px\left( \frac{g(x,r,\cdot)}{\lambda^{v(\cdot)}} \right) \leq 1\right\} \leq a+\varepsilon.
$$
Therefore $b \leq a+\varepsilon$, from which the claim follows by the arbitrariness of $\varepsilon>0$.
\end{proof}

Using Lemma~\ref{lem:inf-sup} with $v(y)\equiv 1$ and $g(x,r,y)=r^{\frac{n}{u(x)}-\frac{n}{p(x)}} f(y)\,\chi_{B(x,r)}(y)$, we get

\begin{corollary}\label{cor:norm-morrey-alt}
Let $p,u\in \PP$ with $p(x)\leq u(x)$. For any {$f\in\mathcal{M}(\Rn)$} it holds
\begin{equation}\label{eq:norm-morrey-alt}
\big\| f \,| M_{\px,\ux}\big\|= \inf\left\{\lambda>0: \sup_{x\in\Rn,r>0} \varrho_\px\left( \tfrac{1}{\lambda}\,r^{\frac{n}{u(x)}-\frac{n}{p(x)}} f\,\chi_{B(x,r)} \right) \leq 1\right\}.
\end{equation}
\end{corollary}

The identification in \eqref{eq:norm-morrey-alt} was already observed in \cite[Lemma~3]{AlmHS08} in the case $1\leq p(x)\leq p^+<\infty$.

\begin{remark}\label{rem:morrey-alternative}
Note that \eqref{eq:norm-morrey-alt} suggests an alternative way of introducing variable exponent Morrey spaces. Instead of presenting the Morrey quasinorm \eqref{def:morrey-norm} directly, one can introduce it from an appropriate semimodular space setting. In fact, it can be checked that
$$
\varrho_{\px,\ux}(f):= \sup_{x\in\Rn,r>0} \varrho_\px\left(r^{\frac{n}{u(x)}-\frac{n}{p(x)}} f\,\chi_{B(x,r)} \right)
$$
defines a left-continuous semimodular in $\mathcal{M}_0(\Rn)$ and that $M_{\px,\ux}$ coincides with the space built according to \eqref{def:modular-space}.

\end{remark}

\begin{lemma}\label{lem:sup}
Let $p\in \PP$. For any $g\in\mathcal{M}(\Rn)$,
$$\sup_{x\in\Rn,r>0} \varrho_\px\left(g\,\chi_{B(x,r)} \right) = \varrho_\px\left(g \right).$$
\end{lemma}

\begin{proof}
Since $\phi_{p(y)}\big(|g(y)|\,\chi_{B(0,N)}(y)\big)$ increases with $N\in\N$, an application of the monotone convergence theorem yields
\begin{eqnarray*}
\sup_{x\in\Rn,r>0} \varrho_\px\left(g\,\chi_{B(x,r)} \right) & \geq & \sup_{N\in\N} \varrho_\px\left(g\,\chi_{B(0,N)} \right) = \sup_{N\in\N} \int_{\Rn} \phi_{p(y)}\big(|g(y)|\,\chi_{B(0,N)}(y)\big)dy\\
& = & \lim_{N\to\infty} \int_{\Rn} \phi_{p(y)}\big(|g(y)|\,\chi_{B(0,N)}(y)\big)dy = \varrho_\px\left(g \right).
\end{eqnarray*}
The converse inequality is clear since $|g|\,\chi_{B(x,r)} \leq |g|$ for any $x\in\Rn$ and $r>0$.
\end{proof}

By \eqref{eq:norm-morrey-alt} and Lemma~\ref{lem:sup} we obtain the following coincidence:

\begin{lemma}\label{lem:coincides}
For any $p\in \PP$ we have $M_{\px,\px} = L_\px$ (with equal quasinorms).
\end{lemma}

\subsection{Mixed Lebesgue-sequence spaces}
To deal with variable exponent Besov and Triebel--Lizorkin scales we need to consider appropriate mixed sequences spaces. For the variable Triebel-Lizorkin scale one can easily define the space $L_\px(\ell_\qx)$ for every $p,q\in\PP$ through the quasinorm
\begin{equation} \label{def:lpq}
\|(f_\nu)_\nu\,| L_\px(\ell_\qx) \| := \big\|
\|(f_\nu(x))_\nu\,| \ell_{q(x)}\|\,| L_\px\big\|\,,
\end{equation}
on sequences {$(f_\nu)_{\nu}\subset \mathcal{M}(\Rn)$ where \eqref{def:lpq} is finite} (cf. \cite{DieHR09}). This is always a norm if $\min\{p^-,q^-\} \geq 1$. Note that $\ell_{q(x)}$ is just a standard discrete Lebesgue space (for each $x\in\Rn$), and that \eqref{def:lpq} is well defined since $q(x)$ does not depend on $\nu$ and the function $x\mapsto \|(f_\nu(x))_\nu\,| \ell_{q(x)}\|$ is always measurable when $q\in\PP$.

The situation is much harder in the variable Besov scale due to the dependence of $q$ on the space variable $x$. Nevertheless, in \cite[Definition~3.1]{AlmH10} the authors were able to introduce the \textit{mixed Lebesgue-sequence space} $\ell_\qx(L_\px)$ within the setting of semimodular spaces as follows. For $p,q\in\PP$, the functional
\begin{equation}\label{def:lpqmod}
\varrho_{\ell_\qx(L_\px)}\big( (f_\nu)_\nu\big ) := \sum_{\nu\ge 0}
\inf\Big\{\lambda>0: \, \varrho_{\px}\Big(f_\nu
/\lambda^{\frac1{\qx}}\Big)\le 1 \Big\}
\end{equation}
defines a left-continuous semimodular on $\mathcal{M}_0(\Rn)$ (with the convention $\lambda^{\frac{1}{\infty}}=1$). Note that if $q^+<\infty$ then \eqref{def:lpqmod} takes the simpler form
\begin{equation}\label{def:lpqmodsimple}
\varrho_{\ell_\qx(L_\px)}\big( (f_\nu)_\nu\big ) = \sum_{\nu\ge 0} \Big\|
|f_\nu|^{\qx} | L_{\frac{\px}{\qx}}\Big\|.
\end{equation}
When $q^+=\infty$ and $p(x)\geq q(x)$ a.e. we can still use this simpler form of \eqref{def:lpqmod}, with the interpretation $\frac{\infty}{\infty}=1$ and the $\qx$-power inside the $L_{\frac{\px}{\qx}}$-norm understood as $\phi_\qx\left(|f_\nu|\right)$ according to \eqref{def:Lpmod-aux}, cf. \cite[Remark~1]{KemV13}.

The space $\ell_\qx(L_\px)$ consists of all sequences $(f_\nu)_{\nu}$  such that $\varrho_{\ell_\qx(L_\px)}\big( \mu(f_\nu)_\nu\big ) < \infty$ for some $\mu>0$.  In \cite{AlmH10} it was shown that
\begin{equation}\label{def:lpqnorm}
  \|(f_\nu)_\nu\,| \ell_\qx(L_\px)\| :=
  \inf\Big\{ \mu>0:\, \varrho_{\ell_\qx(L_\px)}\big( \tfrac1\mu  (f_\nu)_\nu\big ) \le 1\Big\}
\end{equation}
defines a quasinorm in $\ell_\qx(L_\px)$ for every $p,q\in\PP$ and that $\|\cdot\,| \ell_\qx(L_\px)\|$ is a
norm either when $q\geq 1$ is constant and $p^-\geq 1$, or when $\frac{1}{p(x)}+\frac{1}{q(x)} \leq 1$ almost everywhere. More recently, it was
 observed in \cite{KemV13} that it also becomes a norm if $1\leq q(x)\leq p(x)\leq \infty$. Contrarily to the situation when $q$ is constant, the expression \eqref{def:lpqnorm} is not necessarily a norm when $\min\{p^-,q^-\} \geq 1$ (see \cite{KemV13} for an example showing that the triangle inequality may fail in this case).

It is not hard to check that $\|(f_\nu)_\nu\,| \ell_\qx(L_\px)\|<\infty$ implies $(f_\nu)_\nu\in\ell_\qx(L_\px)$, which in turn implies $f_\nu\in L_\px$ for each $\nu\in\Nz$. Note also that the left-continuity of the semimodular ensures the useful equivalence
$$\|(f_\nu)_\nu\,| \ell_\qx(L_\px)\| \le 1 \ \ \ \text{if and only if} \ \ \ \varrho_{\ell_\qx(L_\px)}\big( (f_\nu)_\nu\big ) \le 1 \ \ \ \ \ \text{(unit ball property)}.$$

It is worth noting that $\ell_\qx(L_\px)$ is a really iterated space when $q\in(0,\infty]$ is constant (\cite[Proposition~3.3]{AlmH10}), and in that case the quasinorm is given by
\begin{equation}\label{iterated}
\|(f_\nu)_\nu\,| \ell_q(L_\px)\| = \big\| \big(\| f_\nu\,| L_\px\|\big)_\nu \,| \ell_q\big\|.
\end{equation}

As shown in \cite[Example~3.4]{AlmH10}, the values of $q$ have no influence on $\|(f_\nu)_\nu\,| \ell_\qx(L_\px)\|$ when we restrict ourselves to sequences having just one non-zero entry. In fact, as in the constant exponent case, there holds
$\|(f_\nu)_\nu\,|\, \ell_\qx(L_\px)\| = \|f\,|\, L_\px\|$ when $f_{\nu_0}=f$ for some fixed $\nu_0\in\Nz$ and $f_\nu = 0$ for all $\nu \not= \nu_0$.

\section{Variable exponent mixed Morrey-sequence spaces}\label{sec:mixedmorrey}

We introduce new mixed sequence spaces in the variable exponent Morrey setting as follows.

\begin{definition}
Let $p,q,u\in\PP$ with $p(x)\leq u(x)$. Given a sequence {$(f_\nu)_\nu \subset \mathcal{M}(\Rn)$}, we set
\begin{equation}\label{def:lpqumod}
\varrho_{\ell_\qx \left(M_{\px,\ux}\right)}\big( (f_\nu)_\nu\big ) := \sum_{\nu\ge 0} \sup_{x\in\Rn,r>0}
\inf\left\{\lambda>0:  \varrho_{\px}\Big(r^{\frac{n}{u(x)}-\frac{n}{p(x)}}f_\nu \,\chi_{B(x,r)}/\lambda^{\frac1{\qx}}\Big)\le 1 \right\}.
\end{equation}
\end{definition}

\begin{remark}\label{rem:modular-simpler}
Note that the infimum inside the expression \eqref{def:lpqumod} depends in general on $\nu\in\Nz$, $x\in\Rn$ and $r>0$. Therefore, we should take into account this fact when calculating $\varrho_{\ell_\qx \left(M_{\px,\ux}\right)}\big( (f_\nu)_\nu\big )$. When $q^+<\infty$ or $q^+=\infty$ and $p(x)\geq q(x)$ we can simplify \eqref{def:lpqumod} like in the $\ell_\qx\big(L_{\px}\big)$ case:
\begin{equation}\label{def:lpqumod-simple}
\varrho_{\ell_\qx \left(M_{\px,\ux}\right)}\big( (f_\nu)_\nu\big )= \sum_{\nu\geq 0} \sup_{x\in\Rn,r>0} \left\| \phi_\qx\left(r^{\frac{n}{u(x)}-\frac{n}{p(x)}}|f_\nu| \,\chi_{B(x,r)}\right) |\,L_{\frac{\px}{\qx}}\right\|.
\end{equation}
\end{remark}

\begin{definition}
Let $p,q,u\in\PP$ with $p(x)\leq u(x)$.
The mixed Morrey-sequence space $\ell_\qx\big(M_{\px,\ux}\big)$ consists of all sequences {$(f_\nu)_\nu \subset \mathcal{M}(\Rn)$} such that $\varrho_{\ell_\qx \left(M_{\px,\ux}\right)}\big(\mu (f_\nu)_\nu\big )<\infty$ for some $\mu>0$. For $(f_\nu)_\nu\in \ell_\qx\big(M_{\px,\ux}\big)$ we define
\begin{equation}\label{def:lpqunorm}
\big\|(f_\nu)_\nu\,|\,\ell_\qx\big(M_{\px,\ux}\big)\big\| := \inf\left\{\mu>0:  \varrho_{\ell_\qx \left(M_{\px,\ux}\right)}\Big(\tfrac{1}{\mu} (f_\nu)_\nu\Big )\le 1 \right\}.
\end{equation}
\end{definition}

It can be shown that $\ell_\qx\big(M_{\px,\ux}\big)$ is a vector space: the fact that it is closed with respect to the sum follows simultaneously with the proof of the (quasi)triangle inequality of $\big\|\cdot\,|\,\ell_\qx\big(M_{\px,\ux}\big)\big\|$ (cf. Proposition~\ref{pro:reduction} and Theorem~\ref{theo:lpqunorm}); the other properties are simple to prove.

Later we shall prove that $\varrho_{\ell_\qx \left(M_{\px,\ux}\right)}$ defines a semimodular and that $\big\|\cdot\,|\,\ell_\qx\big(M_{\px,\ux}\big)\big\|$ defines a quasinorm in $\ell_\qx\big(M_{\px,\ux}\big)$. Straightforward calculations show that
\begin{equation}\label{t-power}
\big\|(f_\nu)_\nu\,|\, \ell_\qx\big(M_{\px,\ux}\big)\big\|^t  =  \big\|(|f_\nu|^t)_\nu \;|\, \ell_{\qx/t}\big(M_{\px/t,\ux/t}\big)\big\|\,, \ \ \ \ \ \forall t>0.
\end{equation}

Although the expressions in \eqref{def:lpqumod} and \eqref{def:lpqunorm} are quite complicated to deal with in general, in some situations nice simplifications occur. Let us discuss how the semimodular and the quasinorm behave in some special cases. The corresponding properties are listed below as propositions.

As for the variable Lebesgue-sequence space (cf. \cite[Proposition~3.3]{AlmH10}), we shall have a really iterated structure when $q$ is a constant:

\begin{proposition}\label{pro:really-iterated}
Let $p,q,u\in\PP$ with $p(x)\leq u(x)$. If $q\in(0,\infty]$ is constant (almost everywhere), then
\begin{equation}\label{really-iterated}
\big\|(f_\nu)_\nu\,|\,\ell_q\big(M_{\px,\ux}\big)\big\| = \Big\| \big(\big\|f_\nu\,|\,M_{\px,\ux}\big\|\big)_\nu \, |\, \ell_q \Big\|
\end{equation}
for every sequence {$(f_\nu)_\nu \subset \mathcal{M}(\Rn)$}.
\end{proposition}

\begin{proof}
Using \eqref{def:lpqumod-simple} and \eqref{Lp-t-power}, for constant $q\in(0,\infty)$ we have
\begin{equation*}
\varrho_{\ell_q \left(M_{\px,\ux}\right)}\big( (f_\nu)_\nu\big ) = \sum_{\nu\ge 0} \big\|f_\nu \,|\,M_{\px,\ux}\big\|^q = \Big\| \big(\big\|f_\nu\,|\,M_{\px,\ux}\big\|\big)_\nu \, |\, \ell_q \Big\|^q,
\end{equation*}
from which \eqref{really-iterated} follows, taking into account \eqref{def:lpqunorm} and the homogeneity of $M_{\px,\ux}$- and $\ell_q$-quasinorms.

Consider now the case $q=\infty$. In a sense, the arguments we will use are similar to those mentioned in the proof of Proposition 3.3 in \cite{AlmH10}. However, we give details for completeness. Let $s:=\sup_{\nu\in\Nz} \big\|f_\nu\,|\,M_{\px,\ux}\big\|$ and
$$A:=\left\{\mu>0: \sum_{\nu\geq 0} \inf\left\{\lambda>0:  \sup_{x\in\Rn,r>0} \varrho_{\px}\Big(\tfrac{1}{\mu}\,r^{\frac{n}{u(x)}-\frac{n}{p(x)}}f_\nu \,\chi_{B(x,r)}\Big)\le 1 \right\} \right\}.$$
We want to show that $s=\inf A$ (here we are using the result from Lemma~\ref{lem:inf-sup} with, in particular, $v(y)=1/q(y)=0$ a.e.).

We show first that $\inf A \leq s$ by showing that $s+\varepsilon \in A$ (for $s<\infty$, otherwise there is nothing to prove) for any $\varepsilon>0$. By the alternative definition of $M_{\px,\ux}$ given in Remark~\ref{rem:morrey-alternative} and the monotonicity of the expressions involved, we have, for any $\varepsilon>0$,
$$\sup_{x\in\Rn,r>0} \varrho_{\px}\Big(\frac{r^{\frac{n}{u(x)}-\frac{n}{p(x)}}f_\nu \,\chi_{B(x,r)}}{\varepsilon+\|f_\nu |\, M_{\px,\ux}\|}\Big)\le 1\,, \ \ \ \ \forall \nu\in\Nz,$$
and then also
$$\sup_{x\in\Rn,r>0} \varrho_{\px}\Big(\frac{r^{\frac{n}{u(x)}-\frac{n}{p(x)}}f_\nu \,\chi_{B(x,r)}}{\varepsilon+s}\Big)\le 1\,, \ \ \ \ \forall \nu\in\Nz.$$
So, all the infima in $A$ for $\mu=s+\varepsilon$ are zero and hence $s+\varepsilon\in A$.

Now we show that $s\leq \inf A$ (with $\inf A<\infty$, otherwise it is clear). For any $\mu\in A$ one must have that
\begin{equation}\label{aux-mu}
\sup_{x\in\Rn,r>0} \varrho_{\px}\Big(\tfrac{1}{\mu}\,r^{\frac{n}{u(x)}-\frac{n}{p(x)}}f_\nu \,\chi_{B(x,r)}\Big)\le 1\,, \ \ \ \ \forall \nu\in\Nz,
\end{equation}
as otherwise at least one of the infima in $A$ would be infinite and then $\mu$ could not be in $A$. From \eqref{aux-mu} and the alternative definition of $M_{\px,\ux}$, we conclude that $\big\|f_\nu\,|\,M_{\px,\ux}\big\| \leq \mu$ for all $\nu\in\Nz$. Consequently, $s=\sup_{\nu\in\Nz} \big\|f_\nu\,|\,M_{\px,\ux}\big\| \leq \mu$ and hence $s\leq \inf A$.
\end{proof}

The result above justifies the notation $\ell_\qx\big(M_{\px,\ux}\big)$ even when the space is not iterated and hence the mixed structure prevails.


Next we show that the values of $q(x)$ have no influence when we calculate \eqref{def:lpqumod} for sequences with at most one non-zero entry, as it occurs in the constant exponent situation. This fact was already observed in \cite[Example~3.4]{AlmH10} for the Lebesgue-sequence spaces $\ell_\qx(L_\px)$.

\begin{proposition}\label{pro:noqx}
Let $p,q,u\in\PP$ with $p(x)\leq u(x)$. If $f_{\nu_0}=f$ for some {$f\in\mathcal{M}(\Rn)$} and $\nu_0\in\N_0$, and $f_{\nu}=0$ for all $\nu \neq \nu_0$, then
\begin{equation*}
\big\|(f_\nu)_\nu\,|\,\ell_\qx\big(M_{\px,\ux}\big)\big\| = \big\|f\,|\,M_{\px,\ux}\big\|.
\end{equation*}
\end{proposition}

\begin{proof}
Notice that the result is immediate when $q$ is constant, since the space is iterated in that case. For general $q\in\PP$, recalling \eqref{def:lpqunorm}, we evaluate the infimum
\[
a:=\inf\left\{\mu>0: \sup_{x\in\Rn,r>0}
\inf\left\{\lambda>0:  \varrho_{\px}\left(\tfrac{1}{\mu} \,\frac{r^{\frac{n}{u(x)}-\frac{n}{p(x)}}f \,\chi_{B(x,r)}}{\lambda^{1/\qx}}\right)\le 1 \right\}\leq 1\right\}.
\]
By Lemma~\ref{lem:inf-sup} with $v(y)=1/q(y)$ and $g(x,r,y)=\tfrac{1}{\mu} r^{\frac{n}{u(x)}-\frac{n}{p(x)}}f(y) \,\chi_{B(x,r)}(y)$, we have
\[
a=\inf\left\{\mu>0: \inf\left\{\lambda>0:  \sup_{x\in\Rn,r>0} \varrho_{\px}\left(\tfrac{1}{\mu} \,\frac{r^{\frac{n}{u(x)}-\frac{n}{p(x)}}f \,\chi_{B(x,r)}}{\lambda^{1/\qx}}\right)\le 1 \right\}\leq 1\right\}=:\inf A.
\]
We show that the previous infimum equals
\[
b:=\inf \left\{\mu>0: \sup_{x\in\Rn,r>0}\, \varrho_{\px}\left(\tfrac{1}{\mu}\, r^{\frac{n}{u(x)}-\frac{n}{p(x)}}f \,\chi_{B(x,r)}\right)\le 1 \right\}=:\inf B,
\]
which is precisely $\big\|f\,|\,M_{\px,\ux}\big\|$ by Corollary~\ref{cor:norm-morrey-alt}.

Let us prove that $a \geq b$ (where we can assume $a<\infty$ without loss of generality), which is equivalent to prove that $a+\varepsilon \geq b$ for all $\varepsilon>0$. It suffices to show that $a+\varepsilon \in B$ (for arbitrary $\varepsilon>0$). For any $\varepsilon_1>0$, we have
$$
\inf\left\{\lambda>0:  \sup_{x\in\Rn,r>0} \varrho_{\px}\left(\frac{1}{a+\varepsilon_1} \,\frac{r^{\frac{n}{u(x)}-\frac{n}{p(x)}}f \,\chi_{B(x,r)}}{\lambda^{1/\qx}}\right)\le 1 \right\}\leq 1.
$$
Therefore, for any $\varepsilon_2>0$,
$$
\sup_{x\in\Rn,r>0} \varrho_{\px}\left(\frac{1}{a+\varepsilon_1} \, \frac{r^{\frac{n}{u(x)}-\frac{n}{p(x)}}f \,\chi_{B(x,r)}}{(1+\varepsilon_2)^{1/\qx}}\right)\le 1.
$$
Given any $\varepsilon >0$, consider $\varepsilon_1, \varepsilon_2>0$ such that $a+\varepsilon> (a+\varepsilon_1)(1+\varepsilon_2)^{1/q^-}$. Then $a+\varepsilon> (a+\varepsilon_1)(1+\varepsilon_2)^{1/\qx}$ and hence
$$
\sup_{x\in\Rn,r>0} \varrho_{\px}\left(\frac{1}{a+\varepsilon} \, r^{\frac{n}{u(x)}-\frac{n}{p(x)}} f \,\chi_{B(x,r)}\right)\le 1.
$$
This implies $a+\varepsilon \in B$, as claimed.

Next we prove $a\leq b$ by showing that $A \supset B$. Given any $\mu\in B$,
$$
\sup_{x\in\Rn,r>0}\, \varrho_{\px}\left(\tfrac{1}{\mu}\, \frac{r^{\frac{n}{u(x)}-\frac{n}{p(x)}}f \,\chi_{B(x,r)}}{1^{1/\qx}}\right)\le 1.
$$
Therefore
$$
\inf\left\{\lambda>0:  \sup_{x\in\Rn,r>0} \varrho_{\px}\left(\frac{1}{\mu} \frac{r^{\frac{n}{u(x)}-\frac{n}{p(x)}}f \,\chi_{B(x,r)}}{\lambda^{1/\qx}}\right)\le 1 \right\}\leq 1,
$$
and hence $\mu\in A$.
\end{proof}

The next result shows that if $u(x)=p(x)$ almost everywhere, the variable exponent mixed Morrey-sequence space coincides with the mixed space $\ell_\qx(L_\px)$ introduced in \cite{AlmH10}.

\begin{proposition}
If $p,q\in\PP$ then
\begin{equation*}
\varrho_{\ell_\qx\big(M_{\px,\px}\big)}\big( (f_\nu)_\nu\big ) = \varrho_{\ell_\qx(L_\px)}\big( (f_\nu)_\nu\big)
\end{equation*}
for every sequence $(f_\nu)_\nu \subset \mathcal{M}(\Rn)$. Consequently,
$$ \ell_\qx\big(M_{\px,\px}\big) = \ell_\qx(L_\px) \ \ \ \ \ \ \ \mbox{(with equal quasinorms)}.
$$
\end{proposition}

\begin{proof}
The coincidence above follows from Lemmas~\ref{lem:inf-sup} and \ref{lem:sup}:
\begin{eqnarray*}
\varrho_{\ell_\qx\big(M_{\px,\px}\big)}\big( (f_\nu)_\nu\big) & = & \sum_{\nu\geq 0} \sup_{x\in\Rn, r>0} \, \inf\left\{\lambda>0: \varrho_{\px}\left(\frac{f_\nu\,\chi_{B(x,r)}}{\lambda^{1/\qx}}\right)\leq 1\right\}\\
& = & \sum_{\nu\geq 0} \inf\left\{\lambda>0: \sup_{x\in\Rn, r>0} \varrho_{\px}\left(\frac{f_\nu\,\chi_{B(x,r)}}{\lambda^{1/\qx}}\right)\leq 1\right\}\\
& = & \sum_{\nu\geq 0} \inf\left\{\lambda>0: \varrho_{\px}\left(\frac{f_\nu}{\lambda^{1/\qx}}\right)\leq 1\right\}\\
& = & \varrho_{\ell_\qx(L_\px)}\big( (f_\nu)_\nu\big).
\end{eqnarray*}

\end{proof}

It is time to show that \eqref{def:lpqumod} really defines a semimodular.

\begin{theorem}\label{theo:modular}
Let $p,q,u\in\PP$ with $p(x)\leq u(x)$. Then $\varrho_{\ell_\qx \left(M_{\px,\ux}\right)}$ is a left-continuous semimodular. Moreover, it is a modular if $q^+<\infty$.
\end{theorem}

\begin{proof}
Recalling Definition \ref{def:modular}, we need to check properties (i)-(v). Properties (i) and (iii) are clear. Property (ii) can be shown using the result given in Proposition~\ref{pro:noqx}. Indeed, given any $\beta>0$ and any sequence {$(f_\nu)_\nu \subset \mathcal{M}(\Rn)$} such that $\varrho_{\ell_\qx(M_{\px,\ux})}\big( \beta(f_\nu)_\nu\big)=0$, we have
$$ 0= \varrho_{\ell_\qx(M_{\px,\ux})}\big( \beta(f_\nu)_\nu\big) \geq \varrho_{\ell_\qx(M_{\px,\ux})}\big( (0,\ldots,0,\beta f_{\nu_0},0,\ldots)\big) \geq 0$$
for each $\nu_0\in\Nz$. Thus
$$\left\|f_{\nu_0}\,|\, M_{\px,\ux}\right\| = \left\|(0,\ldots,0,f_{\nu_0},0,\ldots)\,|\, \ell_\qx(M_{\px,\ux})\right\| = 0$$
and hence $f_{\nu_0}=0$ almost everywhere. To check property (iv), just use the assumption that $q$ is bounded and the characterization \eqref{def:lpqumod-simple}.

Let us move now to the proof of the left-continuity (v). We give the details in the case $\varrho_{\ell_\qx(M_{\px,\ux})}\big( (f_\nu)_\nu\big)<\infty$ only. The arguments can be adapted to treat the other case.
Suppose there exists $\varepsilon>0$ and a sequence $(f_\nu)_\nu\subset \mathcal{M}(\Rn)$ such that for every $\mu\in(0,1)$ we have
$$ \varrho_{\ell_\qx(M_{\px,\ux})} \big((f_\nu)_\nu\big)- \varrho_{\ell_\qx(M_{\px,\ux})} \big(\mu^\ast(f_\nu)_\nu\big) \geq \varepsilon, \ \ \ \ \text{for some} \ \ \mu^\ast\in(\mu,1).$$
Then, for each $\nu\in\Nz$, we find $x_\nu\in\Rn$ and $r_\nu>0$ such that
\begin{eqnarray*}
\lefteqn{\sum_{\nu\ge 0} \,\inf\left\{\lambda>0:  \varrho_{\px}\Big(r_\nu^{\frac{n}{u(x_\nu)}-\frac{n}{p(x_\nu)}}\,f_\nu \,\chi_{B(x_\nu,r_\nu)}/\lambda^{\frac1{\qx}}\Big)\le 1 \right\} }\\
& \geq & \sum_{\nu\ge 0} \left[ \sup_{x\in\Rn, r>0} \inf\left\{\lambda>0:  \varrho_{\px}\Big(r^{\frac{n}{u(x)}-\frac{n}{p(x)}}\,f_\nu \,\chi_{B(x,r)}/\lambda^{\frac1{\qx}}\Big)\le 1 \right\} - \frac{\varepsilon}{2^{\nu}}\right]\\
& = & \varrho_{\ell_\qx(M_{\px,\ux})} \big((f_\nu)_\nu\big) - 2\varepsilon \geq \varrho_{\ell_\qx(M_{\px,\ux})} \big(\mu^\ast(f_\nu)_\nu\big) - \varepsilon.
\end{eqnarray*}
But this implies
$$\varrho_{\ell_\qx(L_\px)} \left(\Big(r_\nu^{\frac{n}{u(x_\nu)}-\frac{n}{p(x_\nu)}}\,f_\nu \,\chi_{B(x_\nu,r_\nu)}\Big)_\nu\right) \geq \varrho_{\ell_\qx(L_\px)} \left(\mu^\ast\Big(r_\nu^{\frac{n}{u(x_\nu)}-\frac{n}{p(x_\nu)}}\,f_\nu \,\chi_{B(x_\nu,r_\nu)}\Big)_\nu\right) - \varepsilon $$
which would contradict the left-continuity of the semimodular  $\varrho_{\ell_\qx(L_\px)}$ (see \cite[Proposition~3.5]{AlmH10}).
\end{proof}

\begin{remark}
We take the opportunity to fix the formulation of \cite[Proposition~3.5(a)]{AlmH10}. As we can conclude from the explanation above, the claim there holds if $q^+<\infty$.
\end{remark}

For constant exponents the connection between the semimodular and the quasinorm is immediate. That connection is not so clear when we deal with non-constant indices. The following inequality gives some information on the estimation of the functional \eqref{def:lpqunorm} by the semimodular \eqref{def:lpqumod}. This will be useful in Section~\ref{sec:convolution} below.

\begin{lemma}\label{lem:norm-mod}
Let $p,q,u\in\PP$ with $p(x)\leq u(x)$ and $q^-<\infty$. If $q^+<\infty$ or $\varrho_{\ell_\qx \left(M_{\px,\ux}\right)} \big((f_\nu)_\nu\big)>0$, then
\begin{equation*}
\left\|(f_\nu)_\nu\,|\, \ell_\qx\big(M_{\px,\ux}\big)\right\| \le \max \left\{ \varrho_{\ell_\qx \left(M_{\px,\ux}\right)} \big((f_\nu)_\nu\big)^{\frac{1}{q^-}}, \varrho_{\ell_\qx \left(M_{\px,\ux}\right)} \big((f_\nu)_\nu\big)^{\frac{1}{q^+}} \right\}.
\end{equation*}
\end{lemma}

The proof consists essentially in showing that the right-hand side above is, when positive, one of the $\mu$'s in the defining formula~\eqref{def:lpqunorm} and using also the fact that the semimodular being considered here is indeed a modular when $q^+<\infty$ (cf. Theorem~\ref{theo:modular}).


Finally, we will show that \eqref{def:lpqunorm} really defines a quasinorm in the space $\ell_\qx\big(M_{\px,\ux}\big)$. Since the main issue here is the proof of the (quasi)triangle inequality, for the sake of intelligibility we prove the other quasinorm properties in a separate lemma.

\begin{lemma}\label{lem:norm-part1}
Let $p,q,u\in\PP$ with $p(x)\leq u(x)$. Then \eqref{def:lpqunorm} defines a homogeneous modular in {the vector space of all sequences contained in $\mathcal{M}_0(\Rn)$}.
\end{lemma}

\begin{proof}

The homogeneity can be shown in the standard way using a change of variable in the infimum: apart from the obvious case $\alpha=0$,
\begin{eqnarray*}
\left\|\alpha(f_\nu)_\nu\,|\, \ell_\qx\big(M_{\px,\ux}\big)\right\| &= & \inf \left\{\mu>0: \varrho_{\ell_\qx \left(M_{\px,\ux}\right)} \big(\tfrac{1}{\mu}\,|\alpha|(f_\nu)_\nu\big) \leq 1\right\}\\
&= & |\alpha| \, \inf\left\{\tfrac{\mu}{|\alpha|}>0: \varrho_{\ell_\qx \left(M_{\px,\ux}\right)} \big(\tfrac{1}{\mu/|\alpha|}\,(f_\nu)_\nu\big) \leq 1\right\}\\
& = & |\alpha| \, \left\|(f_\nu)_\nu\,|\, \ell_\qx\big(M_{\px,\ux}\big)\right\|.
\end{eqnarray*}
It remains to prove that \eqref{def:lpqunorm} satisfies properties (i)-(iv) of Definition~\ref{def:modular}. Property (i) is easy and property (iii) is clear from the homogeneity. On the other hand, property (ii) follows from (iv), which we proceed to show. If $\left\|(f_\nu)_\nu\,|\, \ell_\qx\big(M_{\px,\ux}\big)\right\| = 0$ then $\varrho_{\ell_\qx \left(M_{\px,\ux}\right)} \Big(\frac{1}{\mu}(f_\nu)_\nu\Big) \leq 1$ for all $\mu>0$. This implies
$$
\varrho_{\ell_\qx(L_\px)} \left(\frac{1}{\mu}\Big(r^{\frac{n}{u(x)}-\frac{n}{p(x)}}\,f_\nu\,\chi_{B(x,r)}\Big)_\nu\right) \leq 1
$$
for all $\mu>0$, $x\in\Rn$ and $r>0$. Hence
$$
\left\|\Big(r^{\frac{n}{u(x)}-\frac{n}{p(x)}}\,f_\nu\,\chi_{B(x,r)}\Big)_\nu\,|\, \ell_\qx(L_\px)\right\| = 0
$$
for every $x\in\Rn$ and $r>0$. Since $\left\|\cdot\,|\, \ell_\qx(L_\px)\right\|$ is a {modular in the vector space of all sequences contained in $\mathcal{M}_0(\Rn)$}, then we have $f_\nu(y)\,\chi_{B(x,r)}(y)=0$ almost everywhere, for each $\nu\in\Nz$, $x\in\Rn$ and $r>0$. Hence $f_\nu=0$ almost everywhere for all $\nu\in\Nz$.
\end{proof}

\begin{proposition}\label{pro:reduction}
Let $p,q\in\PP$ be such that there exists $A\in[1,\infty)$ and $t\in(0,\infty)$ such that
\begin{equation}\label{ineq-triangle}
\big\|(f_\nu)_\nu+(g_\nu)_\nu\,|\,\ell_\qx\big(L_{\px}\big)\big\|^t \leq A \left(\big\|(f_\nu)_\nu\,|\,\ell_\qx\big(L_{\px}\big)\big\|^t + \big\|(g_\nu)_\nu\,|\,\ell_\qx\big(L_{\px}\big)\big\|^t\right),
\end{equation}
for all $(f_\nu)_\nu, (g_\nu)_\nu \in \ell_\qx\big(L_{\px}\big)$. Then, given $u\in\PP$ such that $p(x)\leq u(x)$, \eqref{ineq-triangle} also holds with $\ell_\qx\big(L_{\px}\big)$ replaced by $\ell_\qx\big(M_{\px,\ux}\big)$, for any  $(f_\nu)_\nu, (g_\nu)_\nu \in \ell_\qx\big(M_{\px,\ux}\big)$.
\end{proposition}

\begin{proof}
{It is not hard to check,  using the fact $q^->0$, that $\big\|(f_\nu)_\nu|\,\ell_\qx\big(M_{\px,\ux}\big)\big\|<\infty$ for sequences $(f_\nu)_\nu\in\ell_\qx\big(M_{\px,\ux}\big)$}. Let $(f_\nu)_\nu, (g_\nu)_\nu \in \ell_\qx\big(M_{\px,\ux}\big)$ and denote
$$
\mu_f:=\big\|(f_\nu)_\nu\,|\,\ell_\qx\big(M_{\px,\ux}\big)\big\|^t \ \ \ \ \textrm{and} \ \ \ \ \mu_g:=\big\|(g_\nu)_\nu\,|\,\ell_\qx\big(M_{\px,\ux}\big)\big\|^t.
$$
It is clear {from the previous lemma} that \eqref{ineq-triangle} is trivial if at least one of the previous quantities is zero. So, we assume $\mu_f,\mu_g>0$. Consider any sequences $(x_\nu)_\nu\subset \Rn$ and $(r_\nu)_\nu \subset (0,\infty)$. Let
$$
F_\nu := r_\nu^{\frac{n}{u(x_\nu)}-\frac{n}{p(x_\nu)}}f_\nu\,\chi_{B(x_\nu,r_\nu)} \ \ \ \textrm{and} \ \ \ G_\nu := r_\nu^{\frac{n}{u(x_\nu)}-\frac{n}{p(x_\nu)}}g_\nu\,\chi_{B(x_\nu,r_\nu)}\,, \ \ \ \ \nu\in\Nz.
$$
Then $(F_\nu)_\nu$ and $(G_\nu)_\nu$ belong to $\ell_\qx\big(L_{\px}\big)$, with
$$ \big\|(F_\nu)_\nu\,|\,\ell_\qx\big(L_{\px}\big)\big\|^t \leq \mu_f \, \ \ \ \textrm{and} \ \ \
\big\|(G_\nu)_\nu\,|\,\ell_\qx\big(L_{\px}\big)\big\|^t \leq \mu_g.
$$
Indeed,
\begin{eqnarray*}
\lefteqn{\big\|(F_\nu)_\nu\,|\,\ell_\qx\big(L_{\px}\big)\big\| = \inf\left\{ \mu>0:  \varrho_{\ell_\qx(L_{\px})} \big(\tfrac{1}{\mu}(F_\nu)_\nu\big) \leq 1 \right\}}\\
& = & \inf\left\{ \mu>0: \sum_{\nu\geq 0} \inf\left\{ \lambda>0: \varrho_\px \Big( \tfrac{1}{\mu} \, \frac{r_\nu^{\frac{n}{u(x_\nu)}-\frac{n}{p(x_\nu)}}f_\nu\,\chi_{B(x_\nu,r_\nu)}}{\lambda^{1/\qx}}\Big) \leq 1 \right\} \leq 1 \right\}\\
& \leq & \inf\left\{ \mu>0: \sum_{\nu\geq 0} \sup_{x\in\Rn, r>0} \inf\left\{ \lambda>0: \varrho_\px \Big( \tfrac{1}{\mu} \, \frac{r^{\frac{n}{u(x)}-\frac{n}{p(x)}}f_\nu\,\chi_{B(x,r)}}{\lambda^{1/\qx}}\Big) \leq 1 \right\} \leq 1 \right\}\\
& = & \inf\left\{ \mu>0:  \varrho_{\ell_\qx \left(M_{\px,\ux}\right)} \big(\tfrac{1}{\mu}(f_\nu)_\nu\big) \leq 1 \right\} = \big\|(f_\nu)_\nu\,|\,\ell_\qx\big(M_{\px,\ux}\big)\big\| = \mu_f^{1/t},
\end{eqnarray*}
and analogously for $(G_\nu)_\nu$, $(g_\nu)_\nu$. Hence, by the hypothesis \eqref{ineq-triangle},
$$
\big\|(F_\nu)_\nu+(G_\nu)_\nu\,|\,\ell_\qx\big(L_{\px}\big)\big\|^t \leq A (\mu_f + \mu_g).
$$
Thus
$$
\left\|\frac{(F_\nu)_\nu+(G_\nu)_\nu}{[A (\mu_f + \mu_g)]^{1/t}}\,|\,\ell_\qx\big(L_{\px}\big)\right\| \leq 1
$$
and then, by the unit ball property,
$$
\sum_{\nu\geq 0} \inf\left\{ \lambda>0: \varrho_\px \Big( \frac{r_\nu^{\frac{n}{u(x_\nu)}-\frac{n}{p(x_\nu)}} (f_\nu+g_\nu)\,\chi_{B(x_\nu,r_\nu)}}{[A (\mu_f + \mu_g)]^{1/t} \,\lambda^{1/\qx}}\Big) \leq 1 \right\}\leq 1.
$$
Observe that this implies that each infimum is at most one and{, being valid for any sequences $(x_\nu)_\nu\subset \Rn$ and $(r_\nu)_\nu \subset (0,\infty)$, also implies} that the corresponding
\begin{equation}\label{aux-sup}
\sup_{x\in\Rn, r>0} \inf\left\{ \lambda>0: \varrho_\px \Big( \frac{r^{\frac{n}{u(x)}-\frac{n}{p(x)}} (f_\nu+g_\nu)\,\chi_{B(x,r)}}{[A (\mu_f + \mu_g)]^{1/t}\, \lambda^{1/\qx}}\Big) \leq 1 \right\}
\end{equation}
is finite (and even at most one). Therefore, given any $\varepsilon >0$ and $\nu\in\Nz$, it is possible to find $y_\nu\in\Rn$ and $R_\nu\in(0,\infty)$ such that
$$ \eqref{aux-sup} \leq \inf\left\{ \lambda>0: \varrho_\px \Big( \frac{R_\nu^{\frac{n}{u(y_\nu)}-\frac{n}{p(y_\nu)}} (f_\nu+g_\nu)\,\chi_{B(y_\nu,R_\nu)}}{[A (\mu_f + \mu_g)]^{1/t} \,\lambda^{1/\qx}}\Big) \leq 1 \right\} +\frac{\varepsilon}{2^\nu}.$$
Taking the sum with respect to $\nu$ in both sides, we get, for the $F_\nu$, $G_\nu$ built from $y_\nu$ and $R_\nu$,
$$
\varrho_{\ell_\qx \left(M_{\px,\ux}\right)} \left( \Big(\frac{f_\nu+g_\nu}{[A (\mu_f + \mu_g)]^{1/t}}\Big)_\nu\right) \leq \varrho_{\ell_\qx\big(L_{\px}\big)} \left( \Big(\frac{F_\nu+G_\nu}{[A (\mu_f + \mu_g)]^{1/t}}\Big)_\nu\right) + 2\varepsilon \leq 1 +  2\varepsilon.
$$
By the arbitrariness of $\varepsilon>0$, we conclude that
$$
\varrho_{\ell_\qx \left(M_{\px,\ux}\right)} \left( \Big(\frac{f_\nu+g_\nu}{[A (\mu_f + \mu_g)]^{1/t}}\Big)_\nu\right) \leq 1
$$
and, consequently,
$$
\left\| \Big(\frac{f_\nu+g_\nu}{[A (\mu_f + \mu_g)]^{1/t}}\Big)_\nu \,|\,\ell_\qx\big(M_{\px,\ux}\big) \right\| \leq 1,
$$
from which the desired inequality follows by homogeneity.
\end{proof}

From Lemma~\ref{lem:norm-part1} and Proposition~\ref{pro:reduction} combined with \cite[Theorems~3.6 and 3.8]{AlmH10} and \cite[Theorem~1]{KemV13}, we get the following final result:

\begin{theorem}\label{theo:lpqunorm}
The functional \eqref{def:lpqunorm} defines a quasinorm in the vector space $\ell_\qx\big(M_{\px,\ux}\big)$ for any $p,q,u\in\PP$ with $p(x)\leq u(x)$. Moreover, it induces a norm in the following cases (each one understood for almost every $x\in\Rn$):
\begin{enumerate}
\item $p(x)\ge 1$ and $q\in[1,\infty]$ is constant;
\item $1\leq q(x) \leq p(x) \leq u(x) \leq \infty$;
\item $\frac{1}{p(x)}+\frac{1}{q(x)}\leq 1$.
\end{enumerate}
\end{theorem}

\begin{remark}
 Note that, in particular, the previous theorem summarizes the cases {where we known that} $\big\|\cdot\,|\,\ell_\qx\big(M_{\px,\ux}\big)\big\|$ satisfies the triangle inequality. We would like also to remark that there are exponents $p, q\in\PP$ with $p(x)\geq 1$ and $q(x)\geq 1$ for all $x\in\Rn$ (even for constant $p$) for which $\big\|\cdot\,|\,\ell_\qx\big(M_{\px,\ux}\big)\big\|$ does not satisfy the triangle inequality (cf. \cite[Theorem~2]{KemV13}, with $p(x)=u(x)$).
\end{remark}

\section{A convolution inequality}\label{sec:convolution}

It is known that the Hardy-Littlewood maximal operator is bounded in $L_\px$ if $p^->1$ and $1/p$ is \emph{globally $\log$-H\"older continuous}, in the sense of conditions \eqref{logloc} and \eqref{loginfty} below (see \cite[Chapter~4]{DHHR11} for a detailed discussion including references). This fact has been crucial, for instance, for the development of harmonic analysis in variable exponent Lebesgue spaces. However, the situation is more complicated when we consider the mixed Lebesgue-sequence spaces. In fact, as observed in \cite[Section~4]{AlmH10} (see also \cite{DieHR09} for the vector-valued case) we can not expect the so-called maximal inequality to hold in the spaces $\ell_\qx(L_\px)$ when $q$ is non-constant. Therefore, the mixed sequence space loses an important feature of the iterated space, namely the inheritance of properties from the starting space $L_\px$. We face a similar problem in our setting of mixed Morrey-sequence spaces.

The difficulty described above was successfully overcome, first in \cite{DieHR09} and then in \cite{AlmH10}, through the study of convolutions involving nice kernels, namely the $\eta$-functions defined by
\[
\eta_{\nu,m}(x) := \frac{2^{n\nu}}{(1+2^\nu|x|)^m}\,, \ \ \ \nu\in \Nz, \ \ m>0.
\]
Note that $\eta_{\nu,m}\in L_{1}$ for $m>n$ and the corresponding $L_1$-norm does not depend on $\nu$.

Convolution inequalities with kernels given by the functions above have been heavily used as a replacement of the boundedness of the maximal operator in the mixed Lebesgue-sequence spaces. Consequently, they proved to be key tools for the development of the theory of the Besov and Triebel-Lizorkin scales with all the indices variable (cf. \cite{AlmH10}, \cite{DieHR09}). Such a convolution inequality was particularly hard to obtain for the spaces $\ell_\qx(L_\px)$ when $q$ is variable (such an inequality is immediate from the corresponding result in $L_\px$ only for constant $q$). The following inequality was proved in \cite[Lemma~4.7]{AlmH10}; see also \cite[Lemma~10]{KemV12} for the size condition indicated on $m$.
\begin{lemma}
Let $p,q\in\PPlog$ with $p^-\geq 1$. If $m>n+c_{\log}(1/q)$, then there exists $c>0$ such that
\begin{equation}\label{conv-eta-mixed-lebesgue}
\| (\eta_{\nu,m} \ast f_\nu)_\nu\,|\, \ell_\qx(L_\px) \| \le c\, \|
(f_\nu)_\nu \,|\, \ell_\qx(L_\px) \|
\end{equation}
for all $(f_\nu)_\nu \in \ell_\qx(L_\px)$.
\end{lemma}
Note that the class of exponents $\PPlog$ and the meaning of $c_{\log}(1/q)$ are defined below in Section~\ref{sec:holdercontinuous} by means of \eqref{logloc} and \eqref{loginfty}.

The main goal of this section is to obtain a corresponding convolution inequality for the mixed Morrey-sequence spaces $\ell_\qx\big(M_{\px,\ux}\big)$. First we collect some auxiliary results.

\subsection{The $\log$-H\"{o}lder continuity and auxiliary results}\label{sec:holdercontinuous}

We say that a continuous function $g\,:\, \Rn\to \R$ is {\em locally $\log$-H\"{o}lder
continuous}, abbreviated $g \in C^{\log}_\loc(\Rn)$, if there exists $c_{\log}(g)\geq 0$ such that
\begin{equation}\label{logloc}
    |g(x)-g(y)| \leq \frac{c_{\log}(g)}{\log (e + 1/|x-y|)}\,, \ \ \ \ \ \mbox{for all $x,y\in\Rn$.}
\end{equation}
The function $g$ is said to be {\em globally $\log$-H\"{o}lder continuous},
abbreviated $g \in C^{\log}(\Rn)$, if it is locally $\log$-H\"{o}lder
continuous and there exists $g_\infty \in \R$ and $c_{\infty}(g)\geq 0$ such that
\begin{equation}\label{loginfty}
|g(x) - g_\infty| \leq \frac{c_{\infty}(g)}{\log(e+ |x|)}\,, \ \ \ \ \ \mbox{for all $x\in\Rn$.}
\end{equation}
The notation $\PPlog$ is used for those variable exponents $p\in \PP$ such that $\frac1p \in C^{\log}(\Rn)$. Moreover we consider $\frac{1}{p_\infty}:=\big(\frac{1}{p}\big)_\infty$.


Sometimes we need to estimate norms of characteristic functions on balls. The following result is taken from \cite[Corollary~4.5.9]{DHHR11}.
\begin{lemma}\label{lem:balls}
Let $p\in\PPlog$ with $p(x)\ge 1$. Then
\begin{equation*}
\|\chi_{B(x,r)}\,|\,L_\px\| \approx \left\{\begin{array}{cc}
                                             r^{\frac{n}{p(x)}}\,, & \mbox{if}\quad r\leq 1, \\
                                             r^{\frac{n}{p_\infty}}\,, & \mbox{if}\quad r\geq 1,
                                           \end{array}
                                           \right.
\end{equation*}
with the implicit constants independent of $x\in\Rn$ and $r>0$.
\end{lemma}

The convolution operator behaves well in $L_\px$ for $\log$-H\"older continous exponents if we take radially decreasing integrable kernels (cf. \cite[Lemma~4.6.3]{DHHR11}).

\begin{lemma}\label{lem:convbounded}
Let $p\in\PPlog$ with $p(x)\ge 1$. Let $\psi\in L_1$ and $\psi_\varepsilon(x):=\varepsilon^{-n}\psi(x/\varepsilon)$, for $\varepsilon>0$. Suppose that  $\Psi(x):=\sup_{|y|\ge |x|} |\psi(y)|$ (the radial decreasing majorant of $\psi$) is integrable and $f\in L_\px$. Then
$$\|\psi_\varepsilon \ast f\,|\, L_\px\| \lesssim  \|\Psi\,|\,L_1\| \, \|f\,|\,L_\px\|$$
(where the implicit constant depends only on $n$ and $p$).
\end{lemma}

\begin{remark}\label{rem:conv-eta}
It is not hard to see that the $\eta$-functions above are good convolution kernels in the sense of the previous lemma. Indeed, taking $\psi=\eta_{0,m}$, with $m>n$, then we have $\Psi= \eta_{0,m}\in L_1$. Thus, putting $\psi_\varepsilon = \eta_{\nu,m}$ with $\varepsilon=2^{-\nu}$, we get the inequality
\begin{equation}\label{con-eta-lebesgue}
\|\eta_{\nu,m} \ast f \,|\, L_\px\| \lesssim  \|f\,|\,L_\px\|\,, \ \ \ \ \ f\in L_\px,
\end{equation}
if $m>n$ and $p\in\PPlog$, with $p(x)\ge 1$.
\end{remark}

The following result allow us to move some special weights inside convolutions with $\eta$-functions; see \cite[Lemma~6.1]{DieHR09} and \cite[Lemma~19]{KemV12}.

\begin{lemma}\label{lem:weights}
Let $\alpha$ be a locally $\log$-H\"older continuous function on $\Rn$ and $m\geq 0$. If $l\geq c_{\log}(\alpha)$ then
$$ 2^{\nu\alpha(x)}\,\eta_{\nu,m+l}(x-y) \lesssim 2^{\nu\alpha(y)}\, \eta_{\nu,m}(x-y)$$
with the implicit constant depending only on the function $\alpha$.
\end{lemma}

%


\subsection{A convolution inequality in mixed Morrey-sequence spaces}

The next theorem gives a corresponding inequality to \eqref{conv-eta-mixed-lebesgue} now for variable mixed Morrey-sequence spaces $\ell_\qx\big(M_{\px,\ux}\big)$.

\begin{theorem}\label{theo:conv-eta-Morrey-seq}
Let $p\in\PPlog$ and $q,u\in\PP$ with $1\leq p^-\leq p(x) \leq u(x) \leq \infty$ and $1/q$ locally $\log$-H\"older continuous. If
$$m>n+c_{\log}(1/q) + n\,\max\Big\{0,\sup_{x\in\Rn}\big(\tfrac{1}{p(x)}-\tfrac{1}{u(x)}\big)-\tfrac{1}{p_\infty}\Big\}$$ then there exists $c>0$ such that
\begin{equation}\label{conv-eta-mixed-morrey}
\left\| (\eta_{\nu,m} \ast f_\nu)_\nu\,|\, \ell_\qx\big(M_{\px,\ux}\big) \right\| \leq c\, \left\|
(f_\nu)_\nu \,|\, \ell_\qx\big(M_{\px,\ux}\big) \right\|.
\end{equation}
for all sequences $(f_\nu)_\nu \in \ell_\qx\big(M_{\px,\ux}\big)$.
\end{theorem}

\begin{proof}
We assume first that $q^-<\infty$ (the case $q=\infty$ will be detailed in the last part of the proof).

\underline{Step 1}:
Let $x_0\in \Rn$ and $r_0>0$ be (arbitrarily) fixed. We prove that inequality \eqref{conv-eta-mixed-morrey} follows if we show that there exists $c_0\in(0,1]$ (independent of $x_0$, $r_0$, $\nu$ and $(f_\nu)_\nu$) such that
\begin{eqnarray}
\lefteqn{\inf\left\{\lambda>0: \, \varrho_{\px}\left( \frac{c_0\,r_0^{\frac{n}{u(x_0)}-\frac{n}{p(x_0)}} (\eta_{\nu,m} \ast f_\nu) \, \chi_{B(x_0,r_0)} }{\lambda^{\frac1{\qx}}}\right)\le 1 \right\}}\nonumber\\
& \!\!\!\!\! \le & \sup_{x\in\Rn, r>0}\inf\left\{\lambda>0: \, \varrho_{\px} \left(\frac{r^{\frac{n}{u(x)}-\frac{n}{p(x)}}\, f_\nu \, \chi_{B(x,r)}}{\lambda^{\frac1{\qx}}}\right)\le 1 \right\} \, + \,2^{-\nu}\,, \label{ineq-mod1}
\end{eqnarray}
provided the supremum on the right-hand side is at most one.\\
Let $(f_\nu)_\nu \in\ell_\qx\big(M_{\px,\ux}\big)$ with $\left\|
(f_\nu)_\nu \,|\, \ell_\qx\big(M_{\px,\ux}\big) \right\| \leq 1$. If \eqref{ineq-mod1} holds, passing to the supremum with respect to $x_0$ and $r_0$ and then taking the summation on $\nu$, we get
\begin{equation*}
\varrho_{\ell_\qx \left(M_{\px,\ux}\right)} \left(c_0(\eta_{\nu,m} \ast f_\nu)_\nu \right)\leq \varrho_{\ell_\qx \left(M_{\px,\ux}\right)} \left((f_\nu)_\nu \right) + 2 \leq 3.
\end{equation*}
By Lemma~\ref{lem:norm-mod} we have
$$ \left\|(\eta_{\nu,m} \ast f_\nu)_\nu \,|\, \ell_\qx\big(M_{\px,\ux}\big)\right\|\leq 3^{1/q^-}c_0^{-1}.$$
Thus \eqref{conv-eta-mixed-morrey} follows by homogeneity.

\underline{Step 2}: We prove \eqref{ineq-mod1}.\\ Suppose that $q^+<\infty$ (the case $q^+=\infty$ can be treated in a similar way; technically it is more complicated since we have to work directly with infima properties  instead of $L_{\frac{\px}{\qx}}$ quasinorms as below). Then \eqref{ineq-mod1} can be written as
\begin{equation}\label{ineq-mod2}
\left\|\left| c_0\,r_0^{\frac{n}{u(x_0)}-\frac{n}{p(x_0)}} (\eta_{\nu,m} \ast f_\nu) \, \chi_{B(x_0,r_0)} \right|^\qx | L_{\frac{\px}{\qx}}\right\|
 \le \sup_{x\in\Rn, r>0}\left\|\left| r^{\frac{n}{u(x)}-\frac{n}{p(x)}} \,f_\nu \,\chi_{B(x,r)} \right|^\qx | L_{\frac{\px}{\qx}}\right\| \, + \,2^{-\nu}
\end{equation}
(assuming the supremum on the right-hand side be at most one). For each $\nu\in\Nz$, let
$$\delta_\nu := \sup_{x\in\Rn\atop r>0}\left\|\left| r^{\frac{n}{u(x)}-\frac{n}{p(x)}} \,f_\nu \,\chi_{B(x,r)} \right|^\qx | \,L_{\frac{\px}{\qx}}\right\| \, + \,2^{-\nu}.$$
Then we want to show that there exists $c_0\in(0,1]$ (as indicated above) such that
\begin{equation*}
\left\|\left| c_0\,\delta_\nu^{-\frac{1}{\qx}}\, r_0^{\frac{n}{u(x_0)}-\frac{n}{p(x_0)}} (\eta_{\nu,m} \ast f_\nu) \, \chi_{B(x_0,r_0)} \right|^\qx | L_{\frac{\px}{\qx}}\right\| \le 1.
\end{equation*}
By the unit ball property the latter is equivalent to
\begin{equation}\label{ineq-mod3}
\varrho_{\frac{\px}{\qx}}\left(\left| c_0\,\delta_\nu^{-\frac{1}{\qx}}\, r_0^{\frac{n}{u(x_0)}-\frac{n}{p(x_0)}} (\eta_{\nu,m} \ast f_\nu) \, \chi_{B(x_0,r_0)}\right|^\qx \right)
 \le 1.
\end{equation}
To obtain \eqref{ineq-mod3} we decompose each $f_\nu$, $\nu\in\Nz$, into the sum
$$
f_\nu=f_\nu^0+\sum_{i=1}^\infty f_\nu^i\,,
$$
where
$$
f_\nu^0:=f_\nu\,\chi_{B(x_0,2r_0)} \ \ \ \ \ \ \textrm{and} \ \ \ \ \ \ f_\nu^i:=f_\nu\,\chi_{B(x_0,2^{i+1}r_0)\setminus B(x_0,2^ir_0)}\,, \ \ \ i\in\N.
$$
Then we have
\begin{equation}\label{aux3}
\Big| \delta_\nu^{-\frac{1}{q(x)}}(\eta_{\nu,m} \ast f_\nu)(x)\Big| \leq \sum_{i=0}^\infty \delta_\nu^{-\frac{1}{q(x)}} \big(\eta_{\nu,m} \ast |f_\nu^i|\big)(x) \lesssim  \sum_{i=0}^\infty \big(\eta_{\nu,m'} \ast \big| \delta_\nu^{-\frac{1}{q(\cdot)}}f_\nu^i\big|\big)(x).
\end{equation}
The second inequality follows from Lemma~\ref{lem:weights} with $l= c_{\log}(1/q)$ and $m':=m-l>n$. Let us give some details on the application of this lemma. Due to the assumption on the right-hand side of \eqref{ineq-mod2}, we have $2^{-\nu} \leq \delta_\nu \leq 1+2^{-\nu}$ for every $\nu\in\N$. If $\frac{\log_2\delta_\nu}{\nu} \in[0,1]$ we use Lemma~\ref{lem:weights} with $\alpha=-1/q$ (observing that $c_{\log}(-1/q)= c_{\log}(1/q)\geq 0$) to get
$$
\delta_\nu^{-\frac{1}{q(x)}}= \left(2^{-\frac{\nu}{q(x)}}\right)^{\frac{\log_2\delta_\nu}{\nu}} \lesssim \left(2^{-\frac{\nu}{q(y)}}\big(1+2^{\nu}|x-y|\big)^{c_{\log}(-1/q)}\right)^{\frac{\log_2\delta_\nu}{\nu}} \lesssim \delta_\nu^{-\frac{1}{q(y)}} \big(1+2^{\nu}|x-y|\big)^{c_{\log}(1/q)}.
$$
The same argument works when $\frac{\log_2\delta_\nu}{\nu} \in[-1,0]$ (choosing this time $\alpha=1/q$). The case $\nu=0$ can be handled in a similar way using that $\log_2\delta_0 \in[0,1]$.

Returning to \eqref{aux3}, the lattice property and the triangle inequality in $L_\px$ yield
$$\left\| \delta_\nu^{-\frac{1}{\qx}} r_0^{\frac{n}{u(x_0)}-\frac{n}{p(x_0)}} (\eta_{\nu,m} \ast f_\nu)\, \chi_{B(x_0,r_0)} \,| L_\px\right\| \lesssim I_1+I_2,$$
where
$$I_1:= \left\| r_0^{\frac{n}{u(x_0)}-\frac{n}{p(x_0)}} \big(\eta_{\nu,m'} \ast \big| \delta_\nu^{-\frac{1}{q(\cdot)}}f_\nu^0\big|\big)\, \chi_{B(x_0,r_0)} \,| L_\px\right\| $$
and
$$I_2:= \left\| \sum_{i=1}^\infty  r_0^{\frac{n}{u(x_0)}-\frac{n}{p(x_0)}} \big(\eta_{\nu,m'} \ast \big| \delta_\nu^{-\frac{1}{q(\cdot)}}f_\nu^i\big|\big)\, \chi_{B(x_0,r_0)} \,| L_\px\right\|. $$
Thus we obtain the desired inequality if we show that $I_1+I_2 \lesssim 1$.

\underline{Step 3}: We show that $I_1 \lesssim 1$ and $I_2 \lesssim 1$.\\
For the first, we apply Lemma~\ref{lem:convbounded} (and Remark~\ref{rem:conv-eta}) and get
\begin{equation*}
I_1 \lesssim \left\| r_0^{\frac{n}{u(x_0)}-\frac{n}{p(x_0)}} \big| \delta_\nu^{-\frac{1}{q(\cdot)}}f_\nu^0\big| \,| L_\px\right\| \lesssim \left\| \delta_\nu^{-\frac{1}{q(\cdot)}}\,(2r_0)^{\frac{n}{u(x_0)}-\frac{n}{p(x_0)}} \big| f_\nu^0\big| \,| L_\px\right\|,
\end{equation*}
which gives the desired estimate taking into account the definition of $\delta_\nu$.

The estimation of $I_2$ is harder. We start by estimating the convolution appearing in that part. For $x\in B(x_0,r_0)$ and $y\in B(x_0,2^{i+1}r_0)\setminus B(x_0,2^{i}r_0)$, we have $|x-y| \geq 2^{i-1}r_0$, and hence
$$\eta_{\nu,m'}(x-y) \leq 2^{\nu n}(1+2^{\nu+i-1}r_0)^{-m'} \leq 2^{\nu n}2^{m'}(1+2^{\nu+i}r_0)^{-m'}.$$
Using this fact and applying H\"older's inequality, we have
$$
\big(\eta_{\nu,m'} \ast \big| \delta_\nu^{-\frac{1}{q(\cdot)}}f_\nu^i\big|\big)(x) \lesssim 2^{\nu n}(1+2^{\nu+i}r_0)^{-m'}\, \|\chi_{B(x_0,2^{i+1}r_0)}\,|\,L_{p'(\cdot)}\|\, \|\delta_\nu^{-\frac{1}{q(\cdot)}}f_\nu^i\,|\,L_\px\|
$$
for every $x\in B(x_0,r_0)$. Therefore,
\begin{eqnarray*}
\lefteqn{I_2 \lesssim \sum_{i=1}^\infty 2^{\nu n}(1+2^{\nu+i}r_0)^{-m'}\,r_0^{\frac{n}{u(x_0)}-\frac{n}{p(x_0)}} (2^{i+1}r_0)^{\frac{n}{p(x_0)}-\frac{n}{u(x_0)}}}\\
& \quad & \times\,\big\|\chi_{B(x_0,2^{i+1}r_0)}\,|\,L_{p'(\cdot)}\big\| \, \Big\|(2^{i+1}r_0)^{\frac{n}{u(x_0)}-\frac{n}{p(x_0)}} \,\delta_\nu^{-\frac{1}{q(\cdot)}}f_\nu^i\,|\,L_\px\Big\| \, \big\|\chi_{B(x_0,r_0)}\,|\,L_\px\big\|.
\end{eqnarray*}
Since the last but one norm is at most one (by the definition of $\delta_\nu$), we have $I_2\lesssim \,I_3$ where
$$I_3:=  \sum_{i=1}^\infty 2^{\nu n}(1+2^{\nu+i}r_0)^{-m'}\,(2^{i+1})^{\frac{n}{p(x_0)}-\frac{n}{u(x_0)}} \, \|\chi_{B(x_0,2^{i+1}r_0)}\,|\,L_{p'(\cdot)}\| \, \|\chi_{B(x_0,r_0)}\,|\,L_\px\|.
$$
In order to estimate $I_3$, we consider two cases.

\noindent \textit{The case $0<r_0\leq 1$}: Choose $J\in\N$ such that $2^{J}r_0\leq 1 < 2^{J+1}r_0$ (consider $J=1$ if no such number exists) and split the sum with respect to $i\in\N$ into two parts, $\sum_{i=1}^{J-1}$ and $\sum_{i=J}^\infty$ (the first sum does not exist if $J=1$). Denote by $I_{3A}$ and $I_{3B}$ the corresponding splitting implied in $I_3$.

If $i\leq J-1$ we have $2^{i+1}r_0 \leq 2^Jr_0 \leq 1$. By Lemma~\ref{lem:balls}
$$
\|\chi_{B(x_0,r_0)}\,|\,L_\px\| \approx r_0^{\frac{n}{p(x_0)}} \ \ \ \ \text{and} \ \ \ \ \|\chi_{B(x_0,2^{i+1}r_0)}\,|\,L_{p'(\cdot)}\| \approx (2^{i+1}r_0)^{\frac{n}{p'(x_0)}}.
$$
Using these estimates, we get
$$I_{3A} \leq \sum_{i=1}^{J-1} 2^{\nu n}(1+2^{\nu+i}r_0)^{-m'}\,r_0^{n}\, (2^{i+1})^{n-\frac{n}{u(x_0)}}.$$
For $\nu\in\{0,1\}$ we use $(1+2^{\nu+i}r_0)^{-m'}\leq 1$ to obtain
$$I_{3A} \lesssim r_0^{n}\, \sum_{i=1}^{J-1} 2^{i\,n} \lesssim (2^Jr_0)^{n} \lesssim 1.$$
For $\nu\geq 2$ we have $2^\nu \geq 4 > (2^{J-1}r_0)^{-1}$ (recall that $2^{J+1}r_0 >1$), which implies $J-1>-\nu -\log_2 r_0$. We consider two separate cases depending on the size of $2^\nu r_0$:\\
The case $\underline{2^\nu r_0 > 1/2}$: here we use that $(1+2^{\nu+i}r_0)^{-m'}\leq (2^{\nu+i}r_0)^{-m'}$ and $m'>n$, and obtain
\begin{equation*}
I_{3A} \lesssim 2^{\nu n} \,r_0^{n}\, \sum_{i=1}^{J-1} (2^{\nu+i}r_0)^{-m'} (2^i)^{n-\frac{n}{u(x_0)}} \, 2^{n-\frac{n}{u(x_0)}}\lesssim \,2^{\nu(n-m')} \,r_0^{n-m'}\, \sum_{i=1}^{\infty} 2^{i(n-m')} \lesssim (2^{\nu}r_0)^{n-m'} \lesssim 1.
\end{equation*}
The case $\underline{2^\nu r_0 \leq 1/2}$: here we have $-\nu -\log_2 r_0 \geq 1$ and we can split the summation in $I_{3A}$ as
$$I_{3A} = \sum_{i=1}^{\lfloor-\nu -\log_2 r_0\rfloor} \cdots\, + \sum_{i=\lfloor-\nu -\log_2 r_0\rfloor+1}^{J-1} \cdots =: I_{3A_1} + I_{3A_2}.$$
To handle $I_{3A_1}$ we use again that $(1+2^{\nu+i}r_0)^{-m'}\leq 1$, and get
$$I_{3A_1} \lesssim 2^{\nu n} \,r_0^{n}\, \sum_{i=1}^{\lfloor-\nu -\log_2 r_0\rfloor} (2^{i+1})^n \lesssim 2^{\nu n} \,r_0^{n}\, 2^{n(-\nu -\log_2 r_0)} \approx 1.$$
In order to handle $I_{3A_2}$, we note that $(1+2^{\nu+i}r_0)^{-m'}\leq (2^{\nu+i}r_0)^{-m'}$ and $m'>n$. Then we obtain
\begin{eqnarray*}
I_{3A_2} & \lesssim & 2^{\nu n} \,r_0^{n}\, \sum_{i=\lfloor-\nu -\log_2 r_0\rfloor+1}^{J-1} (2^{\nu+i}r_0)^{-m'} (2^{i+1})^{n-\frac{n}{u(x_0)}}\\
& \lesssim & (2^{\nu}r_0)^{n-m'}\, \sum_{i=\lfloor-\nu -\log_2 r_0\rfloor+1}^{\infty} 2^{i(n-m')} \lesssim (2^{\nu}r_0)^{n-m'}\,2^{(n-m')(-\nu -\log_2 r_0)}= 1.
\end{eqnarray*}

Let now $i\geq J$. Since $2^{i+1}r_0>1$, by Lemma~\ref{lem:balls} we have
$$
\|\chi_{B(x_0,2^{i+1}r_0)}\,|\,L_{p'(\cdot)}\| \approx (2^{i+1}r_0)^{\frac{n}{(p')_\infty}} = (2^{i+1}r_0)^{\frac{n}{(p_\infty)'}}.
$$
We have
\begin{eqnarray*}
I_{3B} & \lesssim & \sum_{i=J}^{\infty} 2^{\nu n}(1+2^{\nu+i}r_0)^{-m'}\,(2^{i+1})^{\frac{n}{p(x_0)}-\frac{n}{u(x_0)}}\, (2^{i+1}r_0)^{n-\frac{n}{p_\infty}}\,r_0^{\frac{n}{p(x_0)}}\\
& \approx & 2^{\nu(n-m')} \, {r_0}^{-m'+n-\frac{n}{p_\infty}+\frac{n}{p(x_0)}} \sum_{i=J}^{\infty} 2^{(i+1)(-m'+\frac{n}{p(x_0)}-\frac{n}{u(x_0)}+n-\frac{n}{p_\infty})}.
\end{eqnarray*}
Since, by hypothesis, $-m'+\sup_{x\in\Rn}\big(\tfrac{n}{p(x)}-\tfrac{n}{u(x)}\big)+n-\tfrac{n}{p_\infty}<0$, the series above converges. Recalling also that $m'>n$ and that here $r_0\leq 1$ and $2^{J+1}r_0>1$, we get
$$I_{3B} \lesssim {r_0}^{-m'+n-\frac{n}{p_\infty}+\frac{n}{p(x_0)}} \; 2^{(J+1)(-m'+\frac{n}{p(x_0)}-\frac{n}{u(x_0)}+n-\frac{n}{p_\infty})} = (2^{J+1}r_0)^{(-m'+\frac{n}{p(x_0)}-\frac{n}{u(x_0)}+n-\frac{n}{p_\infty})} \, r_0^{\frac{n}{u(x_0)}} \leq 1.$$

\noindent \textit{The case $r_0> 1$}: In this case Lemma~\ref{lem:balls} gives
$$
\|\chi_{B(x_0,r_0)}\,|\,L_\px\| \approx r_0^{\frac{n}{p_\infty}} \ \ \ \ \text{and} \ \ \ \ \|\chi_{B(x_0,2^{i+1}r_0)}\,|\,L_{p'(\cdot)}\| \approx (2^{i+1}r_0)^{\frac{n}{p'_\infty}}.
$$
Using again the size conditions $m'>n$ and $-m'+\sup_{x\in\Rn}\big(\tfrac{n}{p(x)}-\tfrac{n}{u(x)}\big)+n-\tfrac{n}{p_\infty}<0$, we have
\begin{eqnarray*}
I_{3} & \lesssim & \sum_{i=1}^{\infty} 2^{\nu n}\,2^{m'} (1+2^{\nu+i+1}r_0)^{-m'}\,(2^{i+1})^{\frac{n}{p(x_0)}-\frac{n}{u(x_0)}}\, (2^{i+1}r_0)^{n-\frac{n}{p_\infty}}\, r_0^{\frac{n}{p_\infty}}\\
& \approx &  2^{\nu(n-m')}\,r_0^{n-m'}\,\sum_{i=1}^{\infty} 2^{(i+1)(-m'+\frac{n}{p(x_0)}-\frac{n}{u(x_0)}+n-\frac{n}{p_\infty})} \lesssim 1.
\end{eqnarray*}

\underline{Step 4}: Now we consider the case $q^-=\infty$ (i.e., $q(x)=\infty$ almost everywhere).\\
In this case we do not have an error term like in \eqref{ineq-mod1}. Instead of such inequality, now we want to show that there exists a constant $c_0\in(0,1]$, not depending on $x_0$, $r_0$, $\nu$ and $(f_\nu)_\nu$, such that
\begin{equation*}
\left\|c_0\,r_0^{\frac{n}{u(x_0)}-\frac{n}{p(x_0)}} (\eta_{\nu,m} \ast f_\nu) \, \chi_{B(x_0,r_0)}\, | \,L_\px \right\|
 \le \sup_{x\in\Rn\atop r>0}\left\| r^{\frac{n}{u(x)}-\frac{n}{p(x)}} \,f_\nu \,\chi_{B(x,r)} \, | \,L_\px\right\|.
\end{equation*}
for all $(f_\nu)_\nu \in \ell_{\infty}\big(M_{\px,\ux}\big)$.
This is equivalent to show that there exists $c_1 \ge 1$ (also independent of $x_0$, $r_0$, $\nu$ and $(f_\nu)_\nu$), such that
\begin{equation}\label{ineq-mod4}
\left\|r_0^{\frac{n}{u(x_0)}-\frac{n}{p(x_0)}} (\eta_{\nu,m} \ast f_\nu) \, \chi_{B(x_0,r_0)}\, | \,L_\px \right\|
 \le c_1 \sup_{x\in\Rn\atop r>0}\left\| r^{\frac{n}{u(x)}-\frac{n}{p(x)}} \,f_\nu \,\chi_{B(x,r)} \, | \,L_\px\right\|.
\end{equation}
for all $(f_\nu)_\nu \in \ell_{\infty}\big(M_{\px,\ux}\big)$.

As in the last part of Step~2, we have
\begin{eqnarray*}
\lefteqn{\left\| r_0^{\frac{n}{u(x_0)}-\frac{n}{p(x_0)}} (\eta_{\nu,m} \ast f_\nu)\, \chi_{B(x_0,r_0)} \,| L_\px\right\| \leq }\\
 & \leq & \left\| r_0^{\frac{n}{u(x_0)}-\frac{n}{p(x_0)}} \big(\eta_{\nu,m} \ast |f_\nu^0|\big)\, \chi_{B(x_0,r_0)} \,| L_\px\right\| + \left\| \sum_{i=1}^\infty  r_0^{\frac{n}{u(x_0)}-\frac{n}{p(x_0)}} \big(\eta_{\nu,m} \ast \big|f_\nu^i\big|\big)\, \chi_{B(x_0,r_0)} \,| L_\px\right\|.
\end{eqnarray*}
The estimation of the last two norms can be done as we did in $I_1$ and $I_2$ in Step~3{, except that now we get them dominated by the right-hand side of \eqref{ineq-mod4} instead of being dominated by constants alone}. Note that Lemma~\ref{lem:weights} is not necessary here, which means that we can proceed just  with $m'=m$ in the calculations.
\end{proof}

\begin{remark}
The arguments used in Step~4 of the proof above work not only for $q^-=\infty$ but also for any constant exponent $q\in(0,\infty]$ (with some minor modifications). We can see from the proof that the big difference between the constant $q$ case and the variable $q$ case appears in Steps~1-2.
\end{remark}

From Theorem~\ref{theo:conv-eta-Morrey-seq} with  constant $q$ (so $c_{\log}(1/q)=0$) combined with Proposition~\ref{pro:noqx} we can establish a convolution inequality for variable exponent Morrey spaces as follows.

\begin{corollary}\label{cor:conv-eta-Morrey}
Let $p\in\PPlog$ and $u\in\PP$ with $1\leq p^-\leq p(x) \leq u(x) \leq \infty$. If {$m>n+ n\,\max\Big\{0,\sup_{x\in\Rn}\big(\tfrac{1}{p(x)}-\tfrac{1}{u(x)}\big)-\tfrac{1}{p_\infty}\Big\}$}, then there exists $c>0$ such that
\begin{equation}\label{conv-eta-Morrey}
\left\| \eta_{\nu,m} \ast f\,|\, M_{\px,\ux} \right\| \leq c\, \left\| f \,|\,M_{\px,\ux} \right\|
\end{equation}
for all $\nu\in\Nz$ and $f \in M_{\px,\ux}$.
\end{corollary}

Note that \eqref{conv-eta-Morrey} extends to variable Morrey spaces the convolution inequality \eqref{con-eta-lebesgue} already known for variable Lebesgue spaces. Note that when $u(x)=p(x)$ we recover the same assumption on $m$ used in \eqref{con-eta-lebesgue}.

\section{Variable exponent Besov-Morrey spaces}\label{sec:besov-morrey}

In the recent years there was an increase of interest in studying so-called smoothness Morrey spaces (with constant exponents). Basically, they are function spaces built on Morrey spaces with additional information on the smoothness of their elements. Let us give some details on the definition of such spaces of Besov type.

The set $\cS$ denotes the usual Schwartz class of infinitely differentiable
rapidly decreasing complex-valued functions on $\Rn$ and $\cS'$
denotes its dual space of tempered distributions. By $\hat{f}$ we denote the Fourier transform of $f\in\cS$ given by
$$
\hat{f}(\xi):= (2\pi)^{-n/2} \int_{\Rn} e^{-i\xi\cdot x} f(x)\,dx, \quad \xi\in\Rn,
$$
and by $\check{f}$ we denote the inverse Fourier transform of $f$. Both transforms are extended to $\cS'$ in the usual way.

A pair $(\check{\varphi},\check{\Phi})$ of functions in $\cS$ is called \emph{admissible} if
\begin{itemize}
\item $\supp {\varphi} \subset \{x\in\Rn: \, \tfrac{1}{2} \le |x| \le 2 \}\, $ and $\, |{\varphi}(x)| \ge c>0$ when $\tfrac{3}{5} \le |x| \le \tfrac{5}{3}$;
\item $\supp {\Phi} \subset \{x\in\Rn: \, |x| \le 2 \}\, $ and $\, |{\Phi}(x)| \ge c>0$ when $ |x| \le \tfrac{5}{3}$.
\end{itemize}
Set $\varphi_0:=\Phi$ and $\varphi_j:=\varphi(2^{-j}\cdot)$ for $j\in\N$. Then $\varphi_j\in\cS$ for all $j\in \Nz$ and
\[
\supp {\varphi}_j \subset \{x\in\Rn: \, 2^{j-1} \le |x| \le 2^{j+1} \} \,, \ \ \ j\in\N.
\]
Accordingly, $\{\varphi_j\}$ constructed as above is called an \emph{admissible system}.

Let $s\in\R$, $0<p\leq u \leq \infty$ and $0<q\leq \infty$. The Besov-Morrey space $\cN^s_{p,u,q}$ is the set of all distributions $f\in\cS'$ such that
$$
\|f\,|\, \cN^s_{p,u,q}\| := \big\| \big(2^{j s}(\varphi_j \hat{f})^{\vee}\big)_j \, |\, \ell_q(M_{p,u})\big\| < \infty.
$$
Hence the Besov-Morrey space $\cN^s_{p,u,q}$ generalizes the usual Besov space $B^s_{p,q}$ (recall that the latter is modeled on $L_p$ spaces instead of Morrey spaces).
The  spaces $\cN^s_{p,u,q}$ were introduced by Kozono and Yamazaki \cite{KY94} motivated by applications to partial differential equations and they were further developed by Mazzucato \cite{Mazz03} in connection with the study of Navier-Stokes equations.

Besov spaces $B^{\sx}_{\px,\qx}$ on $\Rn$ with all parameters variable were introduced in \cite{AlmH10}. More recently, Besov spaces $\B$ with variable integrability have been studied in the general setting of 2-microlocal spaces, \cite{AlmC15a} and \cite{Kem09,Kem10}. Here $\w = (w_j)_{j\in\Nz}$ belongs to the class $\W$, consisting of sequence of positive measurable functions $w_j$ on $\Rn$ such that
\begin{enumerate}
\item[(i)] there exists $c>0$ such that
\begin{equation}\label{aux6}
0<w_j(x) \leq c\,w_j(y)\left(1+2^j|x-y|\right)^\alpha
\end{equation}
for all $j\in\Nz$ and $x,y\in\Rn$;
\item[(ii)] there holds
\[
2^{\alpha_1}\, w_j(x) \leq w_{j+1}(x) \leq 2^{\alpha_2}\, w_j(x)
\]
for all $j\in\Nz$ and $x\in\Rn$.
\end{enumerate}
Then we say that $\w = (w_j)_{j\in\Nz}$ constitutes an \emph{admissible weight sequence}. In the sequel the parameters $\alpha \ge 0$ and $\alpha_1, \alpha_2\in\R$ (with $\alpha_1\leq \alpha_2$) are arbitrary but fixed numbers.  We refer to \cite[Remark~2.4]{Kem08} for some useful properties of class $\W$, and to \cite[Examples~2.4--2.7]{AlmC15a} for a compilation of basic examples of admissible weight sequences.

On the other hand, variable exponent Morrey spaces $M_{\px,\ux}$ have been recently studied by many authors. Then one may ask for a full variable exponent generalization for the spaces $\cN^s_{p,u,q}$ above. Now that we have mixed Morrey-sequence spaces available, we can introduce $2$-microlocal Besov-Morrey spaces with all indices variable. This new scale provides an unification for various function spaces recently introduced in the literature.


\begin{definition} \label{def:BMspaces}
Consider an admissible system $\{\varphi_j\}$. Let $\w=(w_j)_j\in\W$ be admissible weights and let $p,q,u\in\PP$ with $0<p^-\leq p(x)\leq u(x)\leq \infty$. We define $\BN$ as the collection of all $f\in\cS'$ such that
\begin{equation*}
\big \|f\,|\,\BN \big\|:= \big\| (w_j(\varphi_j \hat{f})^{\vee})_j\,|\, \ell_\qx\big(M_{\px,\ux}\big)\big\| < \infty.
\end{equation*}
\end{definition}

In the particular case $w_j(x)=2^{js(x)}$, with $s \in C^{\log}_\loc(\Rn)$, we have Besov-Morrey spaces with variable smoothness and integrability, and we write in that case
$$\BN=\mathcal{N}^{\sx}_{\px,\ux,\qx}.$$

\begin{remark}
In view of Theorem~\ref{theo:lpqunorm} we see that $\BN$ are quasinormed spaces (normed spaces in each one of the cases listed in Theorem~\ref{theo:lpqunorm}). Moreover, as we shall see below (cf. Remark~\ref{rem:independent}) the spaces $\BN$ are independent of the admissible system $\{\varphi_j\}$ taken in its definition, at least when some parameters satisfy regularity assumptions like $\log$-H\"older continuity. This means that different admissible systems should produce equivalent quasinorms in the corresponding spaces.
\end{remark}

%

\begin{remark}
As regards the partial variable generalization  $\mathcal{N}^s_{\px,\ux,q}$ given in \cite{FuXu11} (already mentioned in the Introduction), it contains wrong arguments used in some proofs. For example, they mix results concerning two different approaches for variable Morrey spaces when working in $\Rn$ and state, in their Lemma~2.1, a maximal inequality which is not clear to hold without some restrictions on the parameters for the variable Morrey spaces considered in that paper.  Related to this, in \cite[Corollary~5.1]{GulS13} a maximal inequality is presented, however with some extra restrictions  on the parameters. Still in \cite{FuXu11}, an important vector-valued maximal inequality is given in their Theorem~2.2, but the proof uses their Lemma~2.5, which does not hold in the setting of $\Rn$.
\end{remark}

\section{Maximal functions characterization}\label{sec:PeetreMaxFunc}

For $(\psi_j)_{j} \subset \cS$ we set $\Psi_j:=\hat{\psi}_j \in\cS$. The \emph{Peetre maximal functions} of $f\in\cS'$ are defined, for every $j\in\Nz$ and $a>0$, by
\[
\big( \Psi^\ast_j f\big)_a(x):= \sup_{y\in\Rn} \frac{|(\Psi_j\ast f)(y)|}{1+|2^j(x-y)|^{a}}\,, \ \ \ x\in\Rn.
\]

We consider here the construction starting with two given functions $\psi_0,\psi_1\in\cS$ and
$$ \psi_j(x):= \psi_1(2^{-j+1}x), \quad j\in\N\setminus\{1\}, \quad x\in\Rn.$$
Then we write $\Psi_j=\hat{\psi}_j$ as mentioned above.

The Peetre maximal functions are important tools in the study of properties of various classical functions spaces. The main goal of this section is to present a characterization of the spaces $\BN$ in terms of such functions and, consequently, the independence of their definition with respect to the admissible system considered (cf. Remark~\ref{rem:independent}).

The following technical lemma gives a discrete convolution type inequality in mixed Morrey-sequence spaces. It can be proved by adapting the arguments used in the proof of \cite[Lemma~3.4]{AlmC15a}. For completeness we give the details in the Appendix.

\begin{lemma}\label{lem:discrete-convolution}
Let $p,q,u\in\PP$ with $p(x)\leq u(x)$. Let $\delta>0$. For any sequence $(g_j)_{j\in\Nz}$ of nonnegative measurable functions on $\Rn$, we denote
\begin{equation*}
G_\nu(x):= \sum_{j=0}^\infty 2^{-|\nu-j|\delta} g_j(x), \quad x\in\Rn,\quad \nu\in\Nz.
\end{equation*}
Then it holds
\begin{equation}\label{ineq:discrete-convolution}
\big\| (G_\nu)_\nu\,|\, \ell_\qx\big(M_{\px,\ux}\big)\big\| \lesssim \big\| (g_j)_j\,|\, \ell_\qx\big(M_{\px,\ux}\big)\big\|.
\end{equation}
\end{lemma}

The next theorem compares maximal functions built from different starting functions.

\begin{theorem}\label{thm:different-Peetre}
Let $\w\in\W$ and $p,q,u\in\PP$ with $p(x)\leq u(x)$. Let $a>0$ and $R\in\Nz$ with $R>\alpha_2$. Let further $\psi_0,\psi_1\in\cS$ with
\begin{equation*}
D^{\beta} \psi_1(0)=0 \ \ \ \text{for} \ \ 0\le |\beta| < R,
\end{equation*}
and $\phi_0,\phi_1\in\cS$ with
\begin{equation*}
|\phi_0(x)|>0 \ \ \ \text{on} \ \ \ \{x\in\Rn: |x|\le k\varepsilon\},
\end{equation*}
\begin{equation*}
|\phi_1(x)|>0 \ \ \ \text{on} \ \ \ \big\{x\in\Rn: \varepsilon\le |x|\le 2k\varepsilon\big\}
\end{equation*}
for some $\varepsilon >0$ and $k\in(1,2]$. Then
$$
\big\| \big( \big(\Psi^\ast_j f\big)_a w_j \big)_j\,|\, \ell_\qx\big(M_{\px,\ux}\big)\big\| \lesssim \big\| \big( \big(\Phi^\ast_j f\big)_a w_j \big)_j\,|\, \ell_\qx\big(M_{\px,\ux}\big)\big\|
$$
holds for every $f\in\cS'$.
\end{theorem}

\begin{proof}
We follow the details given in the proof of \cite[Theorem~4.9]{CaeK18}: near the end of that proof the following inequality is stated:
$$
 \big(\Psi^\ast_\nu f\big)_a(x) w_\nu(x) \le c \sum_{j=0}^\infty 2^{-|j-\nu|\delta} \big(\Phi^\ast_j f\big)_a(x) w_j(x), \quad\; x\in\Rn,\quad \nu\in\Nz\quad f\in\cS',
$$
with $\delta:=\min\{1,R-\alpha_2\}$. It was derived with no reference to $p,q$ or $u$ (so, in particular, without any consideration of a concrete space of functions, apart from $\cS'$). Using now Lemma~\ref{lem:discrete-convolution} and the lattice property of $\ell_\qx\big(M_{\px,\ux}\big)$, the desired result follows.
\end{proof}

Below we use the following abbreviation:
$$c_\infty(1/p,1/u):= \max\Big\{0,\sup_{x\in\Rn}\Big(\tfrac{1}{p(x)}-\tfrac{1}{u(x)}\Big)-\tfrac{1}{p_\infty}\Big\}.$$
Note that $c_\infty(1/p,1/u)=0$ when $p(x)=u(x)$ or $p(x)=p$ is constant.

\begin{theorem}\label{thm:same-Peetre}
Let $\w\in\W$. Assume $p\in\PPlog$ and $q,u\in\PP$ with $p(x)\leq u(x)$ and $1/q$ locally $\log$-H\"older continuous. Let $\psi_0,\psi_1\in\cS$ with
\begin{equation*}
|\psi_0(x)|>0 \ \ \ \text{on} \ \ \ \{x\in\Rn: |x|\le k\varepsilon\},
\end{equation*}
\begin{equation*}
|\psi_1(x)|>0 \ \ \ \text{on} \ \ \ \big\{x\in\Rn: \varepsilon\le |x|\le 2k\varepsilon\big\}
\end{equation*}
for some $\varepsilon >0$ and $k\in(1,2]$. If
\begin{equation}\label{ineq-size-a}
a>\alpha+c_{\log}(1/q) + n\,\big(\tfrac{1}{p^-}+c_\infty(1/p,1/u)\big),
\end{equation}
then
$$
\big\| \big( \big(\Psi^\ast_j f\big)_a w_j \big)_j\,|\, \ell_\qx\big(M_{\px,\ux}\big)\big\| \lesssim \big\| \big((\Psi_j\ast f) w_j \big)_j\,|\, \ell_\qx\big(M_{\px,\ux}\big)\big\|
$$
holds for all $f\in\cS'$.
\end{theorem}

\begin{proof}
Now we follow the proof of \cite[Theorem~4.10]{CaeK18} (and the references mentioned there) where almost all the details can be found: for $a>\alpha$ one has
$$
 \Big(\big(\Psi^\ast_\nu f\big)_a(x)\Big)^t \big(w_\nu(x)\big)^t \lesssim \sum_{j=\nu}^\infty 2^{-(j-\nu)(N-a+\alpha_1)t} \Big(\eta_{j,(a-\alpha)t} \ast \big(|\Psi_j\ast f|w_j\big)^t\Big)(x),
$$
with the involved constant independent of $f\in\cS'$, $x\in\Rn$ and $\nu\in\Nz$, and where $t>0$ and $N\in\Nz$ with $N\geq a$ are at our disposal. It is clearly stated in the mentioned proof in \cite{CaeK18} that the above estimate has nothing to do with the consideration of the parameters $p,q,u$ (so, it was obtained without any consideration of a particular space of functions, apart from $\cS'$).

Consider now $a>0$ satisfying \eqref{ineq-size-a}, $N>a+|\alpha_1|$ and $t\in(0,p^-)$ such that still
$$
a>\alpha+c_{\log}(1/q) + n\,\big(\tfrac{1}{t}+c_\infty(1/p,1/u)\big).
$$
Applying first Lemma~\ref{lem:discrete-convolution} with $\delta=(N-a+\alpha_1)t$ and then Theorem~\ref{theo:conv-eta-Morrey-seq} with $m=(a-\alpha)t$ (and $p/t$, $q/t$, $u/t$ instead of $p$, $q$, $u$, respectively), we get
\begin{eqnarray*}
\big\| \big( \big(\big(\Psi^\ast_\nu f\big)_a w_j \big)^t\big)_\nu\,|\, \ell_{\qx/t}\big(M_{\px/t,\ux/t}\big)\big\| & \lesssim & \big\| \big(\eta_{j,(a-\alpha)t} \ast \big(|\Psi_j\ast f|w_j\big)^t\big)_j\,|\, \ell_{\qx/t}\big(M_{\px/t,\ux/t}\big)\big\|\\ & \lesssim &
\big\| \big((|\Psi_j\ast f| w_j)^t \big)_j\,|\, \ell_{\qx/t}\big(M_{\px/t,\ux/t}\big)\big\|.
\end{eqnarray*}
The conclusion follows now from \eqref{t-power}.
\end{proof}

Now we are ready to formulate the main result on the characterization of the spaces $\BN$ in terms of Peetre maximal functions. Its proof works exactly as the proof of \cite[Theorem~4.5]{CaeK18} (given just before the References in \cite{CaeK18}), using now our Theorems~\ref{thm:different-Peetre} and \ref{thm:same-Peetre} above instead of \cite[Theorem~4.9]{CaeK18} and \cite[Theorem~4.10]{CaeK18}, respectively.

\begin{theorem}\label{thm:Peetre-charact}
Let $\w\in\W$. Assume $p\in\PPlog$ and $q,u\in\PP$ with $p(x)\leq u(x)$ and $1/q$ locally $\log$-H\"older continuous. Let $R\in \Nz$ with $R>\alpha_2$ and $\psi_0,\psi_1\in\cS$ with
\begin{equation*}
D^{\beta} \psi_1(0)=0 \ \ \ \text{for} \ \ 0\le |\beta| < R,
\end{equation*}
and
\begin{equation*}
|\psi_0(x)|>0 \ \ \ \text{on} \ \ \ \{x\in\Rn: |x|\le k\varepsilon\},
\end{equation*}
\begin{equation*}
|\psi_1(x)|>0 \ \ \ \text{on} \ \ \ \big\{x\in\Rn: \varepsilon\le |x|\le 2k\varepsilon\big\}
\end{equation*}
for some $\varepsilon >0$ and $k\in(1,2]$. If $a>\alpha+c_{\log}(1/q) + n\,\big(\tfrac{1}{p^-}+c_\infty(1/p,1/u)\big)$, then
$$
\big \|f\,|\,\BN \big\| \approx \big\| \big((\Psi_j\ast f) w_j \big)_j\,|\, \ell_\qx\big(M_{\px,\ux}\big)\big\| \approx \big\| \big( \big(\Psi^\ast_j f\big)_a w_j \big)_j\,|\, \ell_\qx\big(M_{\px,\ux}\big)\big\|
$$
holds for all $f\in\cS'$.
\end{theorem}

\begin{remark}\label{rem:independent}
Notice that the above theorem contains the conclusion that the spaces $\BN$ given in Definition~\ref{def:BMspaces} are independent of the admissible system considered, when $p\in\PPlog$ and $1/q$ is locally $\log$-H\"older continuous. Note also that Theorem~\ref{thm:Peetre-charact} generalizes to the Morrey setting the maximal function characterization \cite[Theorem~3.1]{AlmC15a} already established in the Lebesgue setting. In particular, the assumption $a>\alpha+c_{\log}(1/q) + n\,\big(\tfrac{1}{p^-}+c_\infty(1/p,1/u)\big)$ agrees with the corresponding one in \cite[Theorem~3.1]{AlmC15a} since $c_\infty(1/p,1/u)=0$ when $p(x)=u(x)$.
\end{remark}

\section{Atomic and molecular characterizations}\label{sec:atomic}

In this section we characterize the spaces $\BN$ in terms of atoms and molecules. As a by-product, we establish the embeddings
$$\cS \hookrightarrow \BN \hookrightarrow \cS'$$
and, moreover, we show that $\BN$ are complete spaces.

We use the notation $Q_{jm}$, with $j\in\Nz$ and $m\in\Zn$, for the closed cube with sides parallel to the coordinate axes, centred at $2^{-j}m$ and with side length $2^{-j}$. By $\chi_{jm}$ we denote the corresponding characteristic function. The notation $d\,Q_{jm}$, $d>0$, will stand for the closed cube concentric with $Q_{jm}$ and of side length $d2^{-j}$.

The building blocks we are interested in are defined as follows:

\begin{definition}[Atoms]
Let $K,L\in\Nz$ and $d>1$. For each $j\in\Nz$ and $m\in\Zn$, a $C^K$-function $a_{jm}$ on $\Rn$ is called a $(K,L,d)$-atom (supported near $Q_{jm}$) if
\begin{equation*}
\supp a_{jm} \subset d\,Q_{jm}\,,
\end{equation*}
\begin{equation*}
\sup_{x\in\Rn} |D^\gamma a_{jm}(x)| \le 2^{|\gamma|j}\,, \ \ \text{for} \ \ 0\le |\gamma| \le K,
\end{equation*}
and
\begin{equation}\label{atom-moments}
\int_{\Rn} x^\gamma\, a_{jm}(x)\, dx=0\,, \ \ \text{for} \ \ j\geq 1 \ \ \text{and} \ \ 0\le |\gamma| <L.
\end{equation}
\end{definition}


\begin{definition}[Molecules]
Let $K,L\in\Nz$ and $M>0$. For each $j\in\Nz$ and $m\in\Zn$, a $C^K$-function $[aa]_{jm}$ on $\Rn$ is called a $(K,L,M)$-molecule (concentrated near $Q_{jm}$) if
\begin{equation*}
|D^\gamma [aa]_{jm}(x)| \le 2^{|\gamma|j} (1+2^j|x-2^{-j}m|)^{-M}\,, \ \ \ \ x\in\Rn, \ \ 0\le |\gamma| \le K,
\end{equation*}
and
\begin{equation*}
\int_{\Rn} x^\gamma\, [aa]_{jm}(x)\, dx=0\,, \ \ \text{for} \ \ j\geq 1 \ \ \text{and} \ \ 0\le |\gamma| <L.
\end{equation*}
\end{definition}


\begin{remark}\label{atom-is-molecule}
$\!\!\!$
\begin{enumerate}
\item[(i)] A $(K,0,d)$-atom is an atom for which condition \eqref{atom-moments} is not required (so, no moment conditions are required). A similar interpretation applies to a $(K,0,M)$-molecule.
\item[(ii)] It is easy to check that if $a_{jm}$ is a $(K,L,d)$-atom (so supported near $Q_{jm}$), then, given any $M>0$, $\big(1+\tfrac{d\sqrt{n}}{2}\big)^{-M}a_{jm}$ is a $(K,L,M)$-molecule concentrated near $Q_{jm}$.
\end{enumerate}
\end{remark}

For the atomic and molecular representations of the spaces $\BN$ we need to introduce appropriate sequence spaces.

\begin{definition}
Let $\w\in\W$ and $p,q,u\in\PP$ with $p(x)\leq u(x)$. The set $\bns$  consists of all (complex-valued) sequences $\lambda=(\lambda_{jm})_{j\in\Nz\;\; \atop m\in\Zn}$ such that
\begin{equation*}
\|\lambda \,| \bns\|:= \left\| \left( \sum_{m\in\Zn} \lambda_{jm}\,w_j(2^{-j}m)\, \chi_{jm}(\cdot) \right)_j\, | \ell_\qx(M_{\px,\ux}) \right\| < \infty.
\end{equation*}
\end{definition}

Using standard arguments, it is easy to see that the set $\bns$ become a quasinormed space equipped with the functional above. Furthermore, taking into account the properties of class $\W$, we can use $w_j(x)$ in place of $w_j(2^{-j}m)$ in the expression above, up to equivalent quasinorms.

For short we will write $(\lambda_{jm})$ instead of $(\lambda_{jm})_{j\in\Nz\;\; \atop m\in\Zn}$  when it is clear we are working with the exhibited set of indices. 

Below we shall use the next embedding. Its proof will be given in the Appendix.

\begin{lemma}\label{lem:embedding-qinfty}
For every $\w\in\W$ and $p,q,u\in\PP$, with $p(x)\leq u(x)$, it holds
$$ n^{\textbf{\textit{w}}}_{\px,\ux,\qx} \hookrightarrow n^{\textbf{\textit{w}}}_{\px,\ux,\infty}.$$
\end{lemma}

The next statement shows that any Besov-Morrey function can somehow be written as a linear combination of atoms.


\begin{theorem}\label{thm:atomic1}
Let $\w\in\W$, $p\in\PPlog$ and $q,u\in\PP$ with $p(x)\leq u(x)$ and $1/q$ locally $\log$-H\"older continuous. Further, let $K,L\in\Nz$ and $d>1$. For each $f\in \BN$ there exist $(K,L,d)$-atoms $a_{jm}\in\cS$ , $j\in\Nz$, $m\in\Zn$, and $\lambda(f)\in\bns$ such that
\begin{equation}\label{sum-convergence}
f=\sum_{j=0}^\infty \sum_{m\in\Zn} \lambda_{jm}(f)\, a_{jm} \quad\quad \text{(convergence in $\cS'$)}
\end{equation}
and there exists a constant $c>0$ (independent of $f$) such that
$$\big\| \lambda(f)\,|\,\bns\big\| \leq c\, \big \|f\,|\,\BN \big\|.$$
\end{theorem}

The proof follows exactly the scheme given in the proof of \cite[Theorem~5.5]{CaeK19}, using now Theorem~\ref{thm:Peetre-charact} above instead of \cite[Theorem~4.5]{CaeK18} and directly the lattice property of the Morrey-sequence spaces $\ell_\qx(M_{\px,\ux})$ instead of the corresponding property of the Morrey spaces $M_{\px,\ux}$. In the last part, to prove that the inner sum in \eqref{sum-convergence} converges in $\cS'$ for the regular distribution given by the pointwise sum, we can use the arguments given in second step of the proof of \cite[Theorem~5.6]{CaeK19}, since we have $n^{\textbf{\textit{w}}}_{\px,\ux,\qx} \hookrightarrow n^{\textbf{\textit{w}}}_{\px,\ux,\infty}$ (cf. Lemma~\ref{lem:embedding-qinfty}). Note that the hypothesis on $L$ there was not used and we can consider our atoms to be essentially molecules with the required  $M$ (actually, $M>\alpha+n$ will be enough for that effect).

\begin{corollary}\label{cor:embed-into-sprime}
Let $\w\in\W$, $p\in\PPlog$ and $q,u\in\PP$ with $p(x)\leq u(x)$ and $1/q$ locally $\log$-H\"older continuous. Then it holds
$$
\BN \hookrightarrow \cS'.
$$
\end{corollary}

The proof of this corollary is essentially the same as the proof of \cite[Corollary~5.7]{CaeK19}. We have to replace $\FE$, $\fes$ and \cite[Theorem~5.5]{CaeK19} respectively by $\BN$, $\bns$ and Theorem~\ref{thm:atomic1} above, and use again that $n^{\textbf{\textit{w}}}_{\px,\ux,\qx} \hookrightarrow n^{\textbf{\textit{w}}}_{\px,\ux,\infty}$.

The next statement establishes the completeness of the variable Besov-Morrey spaces.

\begin{corollary}\label{cor:complete}
Let $\w\in\W$, $p\in\PPlog$ and $q,u\in\PP$ with $p(x)\leq u(x)$ and $1/q$ locally $\log$-H\"older continuous. Then the spaces $\BN$ are complete.
\end{corollary}

Since the proof contains many complicated technical details intrinsic to the definition of the semimodular $\varrho_{\ell_\qx \left(M_{\px,\ux}\right)}$, we prefer to write it down in the Appendix and hence continue here to the atomic/molecular representation of our spaces.

The next theorem gives a result in the opposite direction to Theorem~\ref{thm:atomic1} above. We use the standard notation $\sigma_t:=n\big(\tfrac{1}{\min\{1,t\}}-1\big)$.

\begin{theorem}\label{thm:atomic2}
Let $\w\in\W$, $p\in\PPlog$ and $q,u\in\PP$ with $p(x)\leq u(x)$ and $1/q$ locally $\log$-H\"older continuous. Let $\lambda\in\bns$ and $[aa]_{jm}$ be $(K,L,M)$-molecules with $K>\alpha_2$,
\begin{equation}\label{cond:L}
L>-\alpha_1+\max\left\{ \frac{n}{\inf u},\; \sigma_{p^-} +c_{\log}(1/q) + n\, c_\infty(1/p,1/u) \right\}
\end{equation}
and
\begin{equation}\label{cond:M}
M>L+2n+2\alpha +\sigma_{p^-} +c_{\log}(1/q) + n\, c_\infty(1/p,1/u).
\end{equation}
Then
\begin{equation}\label{sum-iterated}
f:= \sum_{j=0}^\infty \sum_{m\in\Zn} \lambda_{jm} \, [aa]_{jm} \quad\quad \text{(convergence in $\cS'$)}
\end{equation}
belongs to $\BN$ and there exists $c>0$, independent of $\lambda$ and $[aa]_{jm}$ (as long as $K,L,M$ are kept fixed), such that
$$\big \|f\,|\,\BN \big\| \leq c\, \big\| \lambda\,|\,\bns\big\|.$$
\end{theorem}

\begin{proof}
We follow the scheme of the proof of \cite[Theorem~5.21]{CaeK19}. Note, however, that here the exponent $u$ needs not be bounded (cf. \cite[Remark~5.16]{CaeK19}).

\underline{Step~1}: First we observe that the iterated sum in \eqref{sum-iterated} converges in $\cS'$. This follows from \cite[Theorem~5.6]{CaeK19}, observing that $\lambda\in \bns\subset n^{\textbf{\textit{w}}}_{\px,\ux,\infty}$. Following the first step of the proof of \cite[Theorem~5.21]{CaeK19} for the set of special hypotheses (70), (71) of that theorem on $L$ and $M$, we have to replace $\FE$ and $\fes$ by $\BN$ and $\bns$, respectively. In our case, the hypotheses on $L$ and $M$ above guarantee that
\begin{equation}\label{cond:L1}
L>-\alpha_1+ \sigma_{p^-} +c_{\log}(1/q) + n\, c_\infty(1/p,1/u)
\end{equation}
and
\begin{equation}\label{cond:M1}
M>L+2n+\alpha +\sigma_{p^-} +c_{\log}(1/q) + n\, c_\infty(1/p,1/u),
\end{equation}
so that it is possible to choose $t\in \big(0,\min\{1,p^-\}\big)$ such that
$$
\alpha+n/t+c_{\log}(1/q) + n\, c_\infty(1/p,1/u) < \min\{M-L-n, L+n+\alpha_1+\alpha\},
$$
and $N$ strictly in between these two quantities. In what follows we assume that $t$ and $N$ have been chosen in such a way.

\underline{Step~2}: It works much as the second and third steps of the proof of \cite[Theorem~5.21]{CaeK19}. Note now that the choices made in Step 1 guarantee that
$$
(N-\alpha)t>n+t\,c_{\log}(1/q) + n\,t\, c_\infty(1/p,1/u).
$$
Then using the lattice property of the mixed Morrey-sequence space $\ell_\qx\big(M_{\px,\ux}\big)$, Lemma~\ref{lem:discrete-convolution} above with $\delta:=\min\{L+n+\alpha_1-(N-\alpha), K-\alpha_2\}$, \eqref{t-power} and Theorem~\ref{theo:conv-eta-Morrey-seq} with $m:=(N-\alpha)t$ (and $p/t$, $q/t$, $u/t$ instead of $p$, $q$, $u$, respectively), we get that
$$\big \|f\,|\,\BN \big\| \lesssim \Big\| \Big(\sum_{m\in\Zn} |\lambda_{jm}|\, w_j(2^{-j}m)\, \chi_{jm}\Big)_j\,\big| \, \ell_\qx\big(M_{\px,\ux}\big)  \Big\| = \big\| \lambda\,|\,\bns\big\|,$$
finishing the proof.
\end{proof}

\begin{remark}\label{rem:alternative-cond}
As in \cite[Theorem~5.21]{CaeK19}, in the situation $\sup\limits_{x\in\Rn} \big(1-\frac{p(x)}{u(x)}\big) < p^-$ (which necessarily holds when $p^-\geq 1$), the conclusion given in Theorem~\ref{thm:atomic2} above also holds under the alternative conditions
\begin{equation*}
L>-\alpha_1+\sigma_{p^-} +c_{\log}(1/q) + n\, c_\infty\big(1/p,1/u,\min\{1,p^-\}\big)
\end{equation*}
and
\begin{equation*}
M>L+2n+2\alpha +\max\{1,2c_{\log}(1/p)\}\,\sigma_{p^-} +c_{\log}(1/q) + n\, c_\infty\big(1/p,1/u,\min\{1,p^-\}\big),
\end{equation*}
with the abbreviation $c_\infty\big(1/p,1/u,r\big):=\max\Big\{0, \sup\limits_{x\in\Rn} \big(1/p(x)-1/u(x)\big)-r/p_\infty\Big\}$. In fact, following the arguments above with \cite[Theorem~5.14 and Remark~5.20]{CaeK19} instead of \cite[Theorem~5.6 and Remark~5.16]{CaeK19}, this set of assumptions also guarantees that \eqref{cond:L1} and \eqref{cond:M1} hold, hence concluding similarly.
\end{remark}

\begin{remark}
Since $(K,L,d)$-atoms are, up to a multiplicative constant, $(K,L,M)$-molecules (cf. Remark~\ref{atom-is-molecule}), Theorems~\ref{thm:atomic1} and \ref{thm:atomic2} (see also Remark~\ref{rem:alternative-cond}) above allow us to obtain both atomic and molecular characterizations for the spaces $\BN$. This generalizes, in particular, the results obtained by the authors in \cite{AlmC16b} for the non-Morrey case (cf. \cite[Corollaries~4.14 and 4.15]{AlmC16b}).
\end{remark}

The same type of arguments as in the proof of \cite[Corollary~5.22]{CaeK19} can now be used to show that the Schwartz space is embedded into the spaces $\BN$. Using Theorem~\ref{thm:atomic2} above and taking advantage of Proposition~\ref{pro:noqx}, we get then

\begin{corollary}
Let $\w\in\W$, $p\in\PPlog$ and $q,u\in\PP$ with $p(x)\leq u(x)$ and $1/q$ locally $\log$-H\"older continuous. Then it holds
$$\cS \hookrightarrow \BN.$$
\end{corollary}

\section*{Appendix}

\subsection{Proof of Lemma~\ref{lem:discrete-convolution}}
We can assume that $p,q \ge 1$, otherwise one can always take $t\in \left(0, \min\{1,p^-,q^-\}\right)$ and use \eqref{t-power} as follows:
\begin{eqnarray*}
\big\| (G_\nu)_\nu\,| \ell_\qx\big(M_{\px,\ux}\big) \big\|^t  &= & \left\| (G_\nu^t)_\nu\,| \ell_{\frac{\qx}{t}}\big(M_{\frac{\px}{t},\frac{\ux}{t}}\big) \right\|
\leq  \left\|  \sum_{j=0}^\infty 2^{-|j-\nu|\delta\,t}\,  g_j^t  \,| \ell_{\frac{\qx}{t}}\big(M_{\frac{\px}{t},\frac{\ux}{t}}\big) \right\| \\
& \lesssim & \left\| (g_j^t)_j  \,| \ell_{\frac{\qx}{t}}\big(M_{\frac{\px}{t},\frac{\ux}{t}}\big) \right\| = \big\| (g_j)_j \,| \ell_\qx\big(M_{\px,\ux}\big) \big\|.
\end{eqnarray*}
So, let us prove Lemma~\ref{lem:discrete-convolution} for any $p,q\ge 1$ (with $u\geq p$). Suppose that
$$\mu:=\big\| (g_j)_j \,| \ell_\qx\big(M_{\px,\ux}\big) \big\| \in (0,\infty)$$
(otherwise there is nothing to show). By the unit ball property of $\varrho_{\ell_\qx \left(M_{\px,\ux}\right)}$, inequality \eqref{ineq:discrete-convolution} will follow from the inequality
\begin{equation*}
\varrho_{\ell_\qx \left(M_{\px,\ux}\right)}\left(\frac{(G_\nu)_\nu}{c\,\mu}\right) \leq 1,
\end{equation*}
for some constant $c>0$ independent of $\mu$, which will be derived next. With the convention $g_j \equiv 0$ for $j<0$ and the abbreviation
$h_{x,r}(\cdot):=r^{\frac{n}{u(x)}-\frac{n}{p(x)}} \chi_{B(x,r)}(\cdot), \; x\in\Rn, r>0$, by the unit ball property of $\varrho_{\px}$ and Minkowski's inequality we get
\begin{eqnarray}
\varrho_{\ell_\qx \left(M_{\px,\ux}\right)}\left(\frac{(G_\nu)_\nu}{c\,\mu}\right) & = & \sum_{\nu \geq 0} \sup_{x\in\Rn,r>0} \inf\left\{\lambda>0: \, \left\|\frac{h_{x,r}(\cdot)\,G_\nu}{c\,\mu\,\lambda^{1/\qx}} \,| L_\px\right\|\le 1 \right\} \nonumber \\
& \leq & \sum_{\nu \geq 0} \sup_{x\in\Rn,r>0} \inf\left\{\lambda>0: \, \sum_{l\in\Z} 2^{-|l|\delta} \left\|\frac{h_{x,r}(\cdot)\,g_{\nu+l}}{c\,\mu\,\lambda^{1/\qx}} \,| L_\px\right\|\le 1 \right\}. \label{inf}
\end{eqnarray}
Let us define
\[
I_{\nu,l}(x,r):= \inf\left\{\lambda>0: \, c^{-1} c(\delta)2^{-|l|\delta/2} \left\|\frac{h_{x,r}(\cdot)\,g_{\nu+l}}{\mu\,\lambda^{1/\qx}} \,| L_\px\right\|\le 1 \right\}
\]
for $\nu\in\Nz$, $l\in\Z$, $x\in\Rn$, $r>0$, where $c(\delta):= \sum_{l\in\Z} 2^{-|l|\delta/2}$. We claim that, for each $\nu\in\Nz$ (and each $x\in\Rn$, $r>0$), the sum $\sum_{l\in\Z} I_{\nu,l}(x,r)$ is not smaller than the infimum in \eqref{inf}. We may assume that this sum is finite. For any $\varepsilon>0$ we have
$$
c^{-1} c(\delta) 2^{-|l|\delta/2} \left\|\frac{h_{x,r}(\cdot)\,g_{\nu+l}}{\mu\,[I_{\nu,l}(x,r) + \varepsilon 2^{-|l|}]^{1/\qx}} \,| L_\px\right\|\le 1,
$$
so that
$$
c^{-1} c(\delta) \sum_{l\in\Z} 2^{-|l|\delta} \left\|\frac{h_{x,r}(\cdot)\,g_{\nu+l}}{\mu\,[I_{\nu,l}(x,r) + \varepsilon 2^{-|l|}]^{1/\qx}} \,| L_\px\right\|\le \sum_{l\in\Z} 2^{-|l|\delta/2}.
$$
Therefore
$$
c^{-1} \sum_{l\in\Z} 2^{-|l|\delta} \left\|\frac{h_{x,r}(\cdot)\,g_{\nu+l}}{\mu\,\Big(\sum_{k\in\Z} [I_{\nu,k}(x,r) + \varepsilon 2^{-|k|}]\Big)^{1/\qx}} \,| L_\px\right\|\le 1,
$$
and hence
$$
\inf\left\{\lambda>0: \, c^{-1} \sum_{l\in\Z} 2^{-|l|\delta} \left\|\frac{h_{x,r}(\cdot)\,g_{\nu+l}}{\mu\,\lambda^{1/\qx}} \,| L_\px\right\|\le 1 \right\} \leq \sum_{k\in\Z} I_{\nu,k}(x,r) + \varepsilon \sum_{k\in\Z} 2^{-|k|}.
$$
The claim follows by the convergence of the second series on the right-hand side and the arbitrariness of $\varepsilon>0$. Now using it in \eqref{inf} and making a convenient change of variables (choosing the constant $c\ge c(\delta)$ and noting that $q\ge 1$), we have
\begin{eqnarray*}
\varrho_{\ell_\qx \left(M_{\px,\ux}\right)}\left(\frac{(G_\nu)_\nu}{c\,\mu}\right) & \leq & \sum_{\nu \geq 0} \sup_{x\in\Rn,r>0} \sum_{k\in\Z} \inf\left\{\lambda>0: \, c^{-1} c(\delta) 2^{-|k|\delta/2} \left\|\frac{h_{x,r}(\cdot)\,g_{\nu+k}}{\mu\,\lambda^{1/\qx}} \,| L_\px\right\|\le 1 \right\}\\
& \leq & \sum_{\nu \geq 0} \sum_{k\in\Z} \sup_{x\in\Rn,r>0}  c^{-1} c(\delta) 2^{-|k|\delta/2} \inf\left\{\sigma>0: \, \left\|\frac{h_{x,r}(\cdot)\,g_{\nu+k}}{\mu\,\sigma^{1/\qx}} \,| L_\px\right\|\le 1 \right\}\\
& = &  \sum_{k\in\Z} c^{-1} c(\delta) 2^{-|k|\delta/2} \sum_{\nu \geq 0} \sup_{x\in\Rn,r>0} \inf\left\{\sigma>0: \, \left\|\frac{h_{x,r}(\cdot)\,g_{\nu+k}}{\mu\,\sigma^{1/\qx}} \,| L_\px\right\|\le 1 \right\}\\
& = &  \sum_{k\in\Z} c^{-1} c(\delta) 2^{-|k|\delta/2} \,\varrho_{\ell_\qx \left(M_{\px,\ux}\right)}\left(\frac{(g_j)_j}{\mu}\right) \leq 1,
\end{eqnarray*}
with the choice $c=c(\delta)^2$ and taking into account our definition of $\mu$. The proof is complete.

\subsection{Proof of Lemma~\ref{lem:embedding-qinfty}}
We use the following auxiliary result which can be of independent interest:
\begin{lemma}\label{lem:aux-embedding-qinfty}
Let $\varrho$ be an increasing and left-continuous semimodular in $\mathcal{M}_0(\Rn)$. Let $q\in\PP$ and $f\in\mathcal{M}_0(\Rn)$. If $\,\inf\left\{ \lambda>0: \varrho \big(\frac{f}{\lambda^{1/\qx}}\big)\leq 1 \right\} \leq 1$, then $\varrho(f)\leq 1$.
\end{lemma}
\begin{proof}
If $\inf\left\{ \lambda>0: \varrho \big(\frac{f}{\lambda^{1/\qx}}\big)\leq 1 \right\} \leq 1$, then $\varrho \big(\frac{f}{\lambda^{1/\qx}}\big)\leq 1$ for every $\lambda>1$. Hence
$$\varrho \big(\frac{f}{\lambda^{1/q^-}}\big)\leq \varrho \big(\frac{f}{\lambda^{1/\qx}}\big)\leq 1$$
for such values of $\lambda$. Since $\frac{1}{\lambda^{1/q^-}} \to 1^-$ as $\lambda\to 1^+$, by the left-continuity of $\varrho$ we get
$$\varrho(f)=\lim_{\lambda\to 1^+} \varrho \Big(\frac{f}{\lambda^{1/q^-}}\Big) \leq 1.$$
\end{proof}
The embedding in Lemma~\ref{lem:embedding-qinfty} follows if we show that
$$ \big\| (f_j)_j \,| \ell_\infty\big(M_{\px,\ux}\big) \big\| \leq \big\| (f_j)_j \,| \ell_\qx\big(M_{\px,\ux}\big) \big\|$$
for every $(f_j)_j \subset \cS'$. The inequality is clear when the quasinorm on the right-hand side equals $0$ or $\infty$. If this is not the case, by homogeneity it is enough to prove that $\big\| (f_j)_j \,| \ell_\infty\big(M_{\px,\ux}\big) \big\| \leq 1$ when such quasinorm is less than or equal to one.

As above, let us use the abbreviation $h_{x,r}(\cdot):=r^{\frac{n}{u(x)}-\frac{n}{p(x)}} \chi_{B(x,r)}(\cdot), \; x\in\Rn, r>0$. By the unit ball property we have $\varrho_{\ell_\qx \left(M_{\px,\ux}\right)} \big((f_j)_j\big) \leq 1$. Thus
$$ \inf\left\{\lambda>0: \, \varrho_\px\Big(\frac{h_{x,r}(\cdot)\,f_j}{\lambda^{1/\qx}} \Big)\le 1 \right\} \leq 1$$
for each $x\in\Rn$, $r>0$ and $j\in\Nz$. By Lemma~\ref{lem:aux-embedding-qinfty} (applied to the semimodular $\varrho_\px$) we also have $\varrho_\px\big(h_{x,r}(\cdot)\,f_j\big)\le 1$. Therefore
$$
\varrho_{\ell_\infty\big(M_{\px,\ux}\big)} \big((f_j)_j\big) = \sum_{j\geq 0}\sup_{x\in\Rn,r>0} \inf\left\{\lambda>0: \, \varrho_\px\big(h_{x,r}(\cdot)\,f_j\big)\le 1 \right\} =0 \leq 1.
$$

\subsection{Proof of Corollary~\ref{cor:complete}}
Let $(f_k)_k$ be a Cauchy sequence in $\BN$. By Corollary~\ref{cor:embed-into-sprime} and the completeness of $\cS'$, there exists $f\in\cS'$ such that $\lim_{k\to \infty} f_k = f$ in $\cS'$. We have to show that the convergence occurs also in $\BN$.

Given $\varepsilon>0$, let $k_0\in\N$ be such that, for $k,l\geq k_0$, $\big\|f_k-f_l\,|\, \BN\big\| < \varepsilon/2$. Together with the unit ball property and Lemma~\ref{lem:inf-sup}, this implies that
\begin{equation*}
\sum_{j\ge 0} A_{j,k,l} := \sum_{j\ge 0} \inf\Big\{\lambda>0: \sup_{x\in\Rn,r>0} \varrho_{\px}\Big(\frac{h_{x,r} \, w_j (\varphi_j\,\widehat{f_k-f_l})^\vee}{\lambda^{\frac{1}{\qx}}\varepsilon/2}\Big)\le 1 \Big\} \leq 1,
\end{equation*}
with $h_{x,r}$ as in the previous proof. For any $k\geq k_0$, we show that, given any $j\in\Nz$, there exists $l_j$ such that
\begin{equation}\label{ineq:complete2}
B_{j,k} := \inf\Big\{\lambda>0: \sup_{x\in\Rn,r>0} \varrho_{\px}\Big(\frac{h_{x,r} \, w_j (\varphi_j\,\widehat{f_k-f})^\vee}{\lambda^{\frac{1}{\qx}}\varepsilon}\Big)\le 1 \Big\} \leq A_{j,k,l}
\end{equation}
for $l\geq l_j$. Assume, on the contrary, that for some $j_0\in\Nz$ and any $l$ there would always  exist $L\geq l$ such that $B_{j_0,k}> A_{j_0,k,L}$. We then could construct the following:
\begin{itemize}
\item for $l=k_0$ fix such a $L$ and call it $l_1$. So, $B_{j_0,k}> A_{j_0,k,l_1}$ and $l_1\geq k_0$;
\item for $l=l_1+1$ fix such a $L$ and call it $l_2$. So, $B_{j_0,k}> A_{j_0,k,l_2}$ and $l_2>l_1$;
\item for $l=l_2+1$ fix such a $L$ and call it $l_3$. So, $B_{j_0,k}> A_{j_0,k,l_3}$ and $l_3>l_2$;
\item $\cdots$
\end{itemize}
In this way we construct sequences $(l_i)_{i\in\N}$ and $\big(A_{j_0,k,l_i}\big)_{i\in\N}$ with $l_i\to \infty$ as $i\to\infty$ and $B_{j_0,k}> A_{j_0,k,l_i}$ for all $i\in\N$. Suppose that $B_{j_0,k} <\infty$. For each $x\in\Rn$ and $r>0$, we have
$$
\frac{h_{x,r} \, w_{j_0} (\varphi_{j_0}\,\widehat{f_k-f_{l_i}})^\vee}{B_{j_0,k}^{\frac{1}{\qx}}\,\varepsilon/2} \longrightarrow \frac{h_{x,r} \, w_{j_0} (\varphi_{j_0}\,\widehat{f_k-f})^\vee}{B_{j_0,k}^{\frac{1}{\qx}}\,\varepsilon/2} \ \ \ \ \text{as} \ \ \ i\to\infty
$$
pointwisely. Since the $L_\px$ semimodular of the left-hand side is at most one and $\varrho_\px$ satisfies Fatou's lemma (cf. \cite[Lemma~2.3.16(a)]{DHHR11}), the corresponding semimodular of the right-hand side is also less than or equal to one. Hence
\begin{equation*}
\sup_{x\in\Rn,r>0} \varrho_\px\left(\frac{h_{x,r} \, w_{j_0} (\varphi_{j_0}\,\widehat{f_k-f})^\vee}{B_{j_0,k}^{\frac{1}{\qx}}\,\varepsilon/2}\right) \leq 1.
\end{equation*}
If $q^-<\infty$, using the fact $\frac{1}{2} \leq \big(\frac{1}{2^{q^-}}\big)^{1/\qx}$ and monotonicity properties, we get
\begin{equation*}
\sup_{x\in\Rn,r>0} \varrho_\px\left(\frac{h_{x,r} \, w_{j_0} (\varphi_{j_0}\,\widehat{f_k-f})^\vee}{\big(\frac{B_{j_0,k}}{2^{q^-}}\big)^{\frac{1}{\qx}}\,\varepsilon}\right) \leq 1.
\end{equation*}
But this contradicts the definition of $B_{j_0,k}$. Note that if $q^-=\infty$, then either $B_{j_0,k}=\infty$ (ruled out here) or $B_{j_0,k}=0$, in which case \eqref{ineq:complete2} is clear.\\
Using similar arguments, it is not hard to check that the case $B_{j_0,k} =\infty$ leads also to a contradiction. Hence inequality \eqref{ineq:complete2} is proved.

Consider now $J\in\N$. From \eqref{ineq:complete2}, for $l\geq \max\{l_0, \ldots l_J\}$ we have
$$ \sum_{j=0}^J B_{j,k} \leq \sum_{j=0}^J A_{j,k,l}.$$
Choosing, in addition, $l\geq k_0$, we also have
$$ \sum_{j=0}^J B_{j,k} \leq \sum_{j=0}^\infty A_{j,k,l} \leq 1,$$
and hence also $\sum_{j=0}^\infty B_{j,k} \leq 1$. This shows that
$$\big\|f_k-f\,|\, \BN\big\| \leq \varepsilon \quad \quad \text{for} \quad k\geq k_0.$$
Finally, we conclude that
$$ f= (f-f_k)+f_k \in \BN \ \ \ \text{and} \ \ \ \lim_{k\to \infty} f_k = f \ \ \ \text{in} \ \ \BN.$$


\end{document}